\documentclass{amsart}

\usepackage[utf8]{inputenc}
\usepackage[T1]{fontenc}
\usepackage[english]{babel}
\usepackage{color}
\usepackage{amsmath} 
\usepackage{amssymb} 
\usepackage{amsthm}
\usepackage{graphicx}
\usepackage[colorlinks=true, linkcolor=blue]{hyperref}
\usepackage{cancel}

\usepackage{lmodern}
\usepackage{microtype}

\theoremstyle{plain}
\newtheorem{thm}{Theorem}[section]
\newtheorem*{thm*}{Theorem}
 
\newtheorem{lem}{Lemma}[section] 
\newtheorem{cor}{Corollary}[section]
\newtheorem*{cor*}{Corollary}

\newtheorem{defi}{Definition}[section]

\newtheorem{rem}{Remark}

\newcommand {\R} {\mathbb{R}} \newcommand {\Z} {\mathbb{Z}}
\newcommand {\T} {\mathbb{T}} \newcommand {\N} {\mathbb{N}}

\newcommand {\p} {\partial}
\newcommand {\dt} {\partial_t}
\newcommand {\sgn} {\text{sgn}}

\usepackage{mathtools}
\newcommand{\supp}{\text{supp}}
\newcommand{\loc}{\text{loc}}

\begin{document}
\title[On mixing scales]{On geometric and analytic mixing scales: comparability
  and convergence rates for transport problems}
\author{Christian Zillinger}
\address{
Department of Mathematics,
University of Southern California,
3620 S. Vermont Avenue,
Los Angeles, CA 90089-2532, US}
\email{zillinge@usc.edu}

\begin{abstract}
  In this article we are interested in the geometric and analytic mixing scales
  of solutions to passive scalar problems.
  Here, we show that both notions are comparable after possibly removing large
  scale projections. In order to discuss our techniques in a transparent way, we
  further introduce a dyadic model problem.

  \noindent
  In a second part of our article we consider the question of sharp decay rates for
  both scales for Sobolev regular initial data when evolving under the transport
  equation and related active and passive scalar equations. Here, we show that slightly
  faster rates than the expected algebraic decay rates are optimal.
\end{abstract}
\maketitle

\tableofcontents

\section{Introduction and Main Results}
\label{sec:intro}
In this article we are interested in the mixing behavior of passive scalar
problems
\begin{align}
  \label{eq:8}
  \begin{split}
    \dt \rho + v \cdot \nabla \rho &=0, \\
    \rho|_{t=0}&=\rho_{0},
  \end{split}
\end{align}
where $v(t)$ is a given divergence-free vector field on $\R^{n}$ or $\T^{n}\times\R^{n}$. Assuming sufficient
regularity of $v$, the flow preserves all $L^{p}$ norms, i.e.
$\|\rho(t)\|_{L^{p}}=\|\rho_{0}\|_{L^{p}}$ for all $p \in [1,\infty]$ and all
$t>0$. However, if the flow is for instance ergodic and $\rho_{0}$ has mean
zero, then $\rho(t)$ weakly converges to zero as time tends to infinity and all
$L^{p}$ norms strictly decrease in this weak limit. The solution is \emph{mixed}
as $t\rightarrow \infty$.

In order to quantify this limiting behavior, one commonly considers two
different functionals:
\begin{defi}[Mixing scales; c.f. \cite{thiffeault2012using} ]
  Let $\rho: \R^{n}\rightarrow \R$ be a given measurable function. Then we call
  $\|\rho\|_{H^{-1}}$ the \emph{analytic mixing scale}.

  Furthermore, for given $r>0$, we define the \emph{geometric mixing
    functionals}
  \begin{align}
    \label{eq:7}
    \mathfrak{g}_{r}[\rho]:=\sup_{B_{R}(\xi); R \geq r} \frac{1}{|B_{R}|} \left| \int_{B_{R}(\xi)} \rho \right|.
  \end{align}
  If further $\rho \in L^{\infty}$, then for each $\kappa \in (0,1)$ we define
  the \emph{geometric mixing scale} as
  \begin{align}
    \mathcal{G}_{\kappa}[\rho]:=\inf \{r: \mathfrak{g}_{r}[\rho]\leq \kappa \|\rho\|_{L^{\infty}}\}.
  \end{align}
\end{defi}

As one of the main results of this article, we show that while both notions are
not equivalent, they are comparable in the sense that smallness of one implies
smallness of the other.
\begin{thm}[Comparison of mixing scales]
  \label{thm:continuous}
  Let $\rho \in L^{2}(\R^{n})$ and $\|\rho\|_{L^{2}}\leq 1$. Then for all
  $0<\epsilon \leq 1$ it holds that:
  \begin{enumerate}
  \item If $\|\rho\|_{H^{-1}}\leq \epsilon$ and $\rho$ is supported in $B_{1}$, then also
    $\mathfrak{g}_{\epsilon'}[\rho]\leq C \epsilon'$ for all $\epsilon'\geq
    \epsilon^{\alpha}$ and $\mathfrak{g}_{\tilde{\epsilon}}[\rho] \leq C$ for
    all $\tilde{\epsilon}\geq \epsilon^{\beta}$, where $\alpha=\frac{2}{n+2}$
    and $\beta=\frac{2}{n+4}$ depend only on the dimension.
    
    In particular, supposing additionally that $\|\rho\|_{L^{\infty}}=1$, it
    follows that
    \begin{align*}
      \mathcal{G}_{C}[\rho]&\leq \tilde{\epsilon}, \\
      \mathcal{G}_{C\epsilon'}[\rho]&\leq \epsilon',
    \end{align*}
  \item If $\mathfrak{g}_{\epsilon}[\rho]\leq \epsilon$ and $\rho$ is supported
    in a compact set $K$, then also $\|\rho\|_{H^{-1}}\leq C_{K} \epsilon$.
  \end{enumerate}
  These estimates are optimal in the powers of $\epsilon$.
\end{thm}

In order to introduce our methods, we construct a dyadic \emph{Walsh-Fourier model} on
$L^{2}(\T)$ in Section \ref{sec:walsh}, where we introduce new dyadic analogues of both scales
and show them to be equivalent when restricted to appropriate subspaces $E_{j}$. In
particular, in that setting optimality of estimates is transparent.
Subsequently, we discuss the continuous case as stated in Theorem
\ref{thm:continuous} in Section \ref{sec:proofs}.

A natural question here, of course, is whether
\begin{align}
  \label{eq:6}
  \mathfrak{g}_{r}[\rho] \leq Cr
\end{align}
can be assumed in applications. Indeed, if $\rho(t)$ solves the passive scalar
problem \eqref{eq:8} and asymptotically converges weakly to a non-trivial state
$\rho_{\infty}$, then we can generally not expect better control than
\begin{align*}
  \mathfrak{g}_{r}[\rho(t)] \leq \|\rho_{\infty}\|_{L^{\infty}}.
\end{align*}
However, as we discuss in Section \ref{sec:examples}, upon removing large scale
projections (corresponding to asymptotic states) this assumption is natural and
comparability holds in the above sense.
\\

As a second part of our article, in Section \ref{sec:transport}, we consider the
evolution of the mixing scales under transport-type equations and are interested
in (sharp) upper and lower bounds on decay rates of the scales. As a first model
problem we consider the case of $\rho(t)$ evolving under the free transport
equation on $\T^{n} \times \R^{n}$:
\begin{align}
  \label{eq:9}
  \begin{split}
    \dt \rho + y \p_{x} \rho &=0 \text{ on } (0,\infty) \times \T^{n} \times \R^{n},\\
    \rho_{t=0}&=\rho_{0} \text{ on } \T^{n} \times \R^{n}.
  \end{split}
\end{align}
Here, we show that if the initial data is normalized in a Sobolev space $H^{s},
0 \leq s \leq 1$, with respect to $y$, then the at first expected decay rates of $t^{-s}$ turn out to
be slightly suboptimal and instead decay rates of $t^{-s}o(1)$ are achieved.

\begin{thm}
  \label{thm:mixing}
  In the following, let $0< s \leq 1$, $u_{0} \in L^{2}(\T^{n}; H^{s}(\R^{n}))$
  with $\int_{\T^{n}} u_{0}(x,y)dx =0$, and let
  \begin{align*}
    u(t,x,y)= u_{0}(t,x-ty,y),
  \end{align*}
  be the solution of the free transport problem.
  For $\sigma, s \in \R$ let $H^{\sigma}H^{s}= H^{\sigma}(\T^{n}; H^{s}(\R^{n}))$ denote the Hilbert space with norm
  \begin{align*}
  \|u\|_{H^{\sigma}H^{s}}^{2}= \sum_{k \in \Z^{n}} \langle k \rangle^{2\sigma} \int_{\R^{n}} \langle \eta \rangle^{2s}|\tilde{u}(k,\eta)|^{2} d\eta. 
  \end{align*}
  \begin{enumerate}
  \item There exists $C_{s}>1$ such that for all $t \geq 1$ and all initial data
    \begin{align*}
      \|u(t)\|_{L^{2}H^{-1}} \leq C t^{-s}\|u_{0}\|_{H^{-s}H^{s}}.
    \end{align*}    
  \item Let $\alpha_{j}>0$ with $\|(\alpha_{j})_{j}\|_{l^{2}}=1$. Then there
    exist $c>0$, a sequence of times $t_{j}\rightarrow \infty$ and initial data
    $u_{0}$ such that 
    \begin{align*}
      \|u(t_{j})\|_{L^{2}H^{-1}} \geq c \alpha_{j} t_{j}^{-s} \|u_{0}\|_{H^{-s}H^{s}}.
    \end{align*}
  \item There exists no non-trivial initial data $u_{0} \in L^{2}(\T^{n}; H^{s}(\R^{n}))$ such that
    \begin{align*}
      \|u(t_{j})\|_{L^{2}H^{-1}} \geq c t_{j}^{-s} \|u_{0}\|_{H^{-s}H^{s}}
    \end{align*}
    along some sequence $t_{j}\rightarrow \infty$.
  \end{enumerate}
\end{thm}

In the second statement, $t_{j}$ can always be chosen larger and more rapidly
increasing. For instance, we may chose $t_{j}=\exp(\exp(\dots \exp(j)))$ and
$\alpha_{j}=\frac{1}{j}=\ln(\ln(\dots \ln(t_{j})))$ as iterated exponentials and
logarithms. Informally stated, the theorem hence shows that algebraic decay
rates can be achieved along a subsequence up to an arbitrarily small loss.
Conversely, the third statement shows that this loss is necessary and that the
lower estimate is sharp in this sense.

We remark that in several works on (linear) inviscid damping,
\cite{Zhang2015inviscid}, \cite{CZZ17}, \cite{Zill5},
\cite{bedrossian2015inviscid} or (linear) Landau damping
\cite{bedrossian2013landau}, it is shown that perturbations to the Euler
equations or Vlasov-Poisson equations scatter to solutions of the free transport
problem as $t \rightarrow \infty$. As a corollary, we hence obtain the
optimality of the decay rates for these equations as well.

\begin{cor}
  Let $U(y)$ be Bilipschitz and $U'' \in W^{2,\infty}$ with
  $\|U''\|_{W^{2,\infty}}$ sufficiently small. Then for any $s \in (0,1)$ and
  any $\omega_{0} \in H^{s}(\T \times \R)$ the solution $\omega$ of the
  linearized Euler equations
  \begin{align*}
    \dt \omega + U(y)\p_{x} \omega - U''(y)\p_{x}\Delta^{-1} \omega &=0, \\
    \omega|_{t=0}&=\omega_{0}
  \end{align*}
  satisfies the statements of Theorem \ref{thm:mixing}.
\end{cor}

\begin{proof}
  In \cite{Zill5}, we have shown that solutions to the linearized Euler
  equations around such a shear flow $U(y)$ scatter in $H^{s}$. That is, for any
  $\omega_{0} \in H^{s}$ there exists $W$ such that the associated solution
  $\omega$ with initial data $\omega_{0}$ satisfies
  \begin{align*}
    \omega(t,x-tU(y),y) \xrightarrow{H^{s}} W
  \end{align*}
  as $ t\rightarrow \infty$. Furthermore, the scattering map $\omega_{0}\mapsto
  W$ is a small perturbation of the identity and hence an isomorphism. Thus,
  both the decay rates and the optimality follow from Theorem \ref{thm:mixing}.
\end{proof}

Concerning more general passive scalar problems, a recent active area of
research, \cite{alberti2014exponential}, \cite{seis2017quantitative},
\cite{crippa2017cellular}, is given by the study of upper and lower bounds on
decay rates of mixing scales for solutions of \eqref{eq:8}
\begin{align*}
  \dt \rho + v \cdot \nabla \rho =0,
\end{align*}
where $v$ may be chosen arbitrarily under given constraints such as
$\|v(t)\|_{W^{1,p}}\leq 1$. Here, out comparison result allows us to obtain an
estimate of analytic mixing costs as a corollary of the results of
\cite{crippa2008estimates} on geometric mixing costs (c.f. Corollary \ref{cor:mixing_cost}).
\begin{cor}
  Let $p>1$ and $\rho|_{t=0}= 1_{[0,1/2]}(x_{2}) \in L^{1}(\T^{2})$  and suppose
  that for $\epsilon>0$ and some $0< \kappa < \frac{1}{2}$ the solution $\rho$
  of
  \begin{align*}
    \begin{split}
      \dt \rho + v\cdot \nabla \rho&=0,\\ \nabla \cdot \rho&=0.
    \end{split}
  \end{align*}
  satisfies
  \begin{align*}
    \left\|\rho|_{t=1}-\frac{1}{2}\right\|_{H^{-1}} \leq \epsilon.
  \end{align*}
  Then for $p>1$ the velocity field $v$ satisfies
  \begin{align*}
    \int_{0}^{1}\|\nabla v \|_{L^{p}} dt \geq C |\log(\epsilon)|.
  \end{align*}
\end{cor}

The remainder of the article is organized as follows:
\begin{itemize}
\item In Section \ref{sec:examples}, we discuss the comparability of both mixing
  scales using two prototypical examples and also discuss the role of large
  scale asymptotic profiles.
  
\item In Section \ref{sec:walsh}, we introduce a new dyadic Walsh-Fourier model of
  mixing scales. Due to improved orthogonality properties here we can establish
  our estimates in a transparent, accessible way.
  
\item Subsequently, in Section \ref{sec:proofs} we show that most properties and
  estimates persist in the continuous setting despite the loss of beneficial
  additional structure.
  
\item Finally, Section \ref{sec:transport} considers (sharp) decay rates of
  mixing scales under passive scalar problems. Here, we first establish optimal
  rates for the free transport problem and then discuss more general dynamics.
\end{itemize}

\section{Preliminaries and Prototypical Examples}
\label{sec:examples}
In \cite{lunasin2012optimal} two families of functions are constructed to
highlight the differences of the analytic mixing scale \eqref{eq:6} and the
geometric mixing functionals \eqref{eq:7}. In order to introduce our ideas, we
recall their construction and show that after removing the large scale weak
limit of the second family both notions are comparable and thus motivate our
choice of spaces and estimates.

\subsection{The Analytic Mixing Scale Controls the Geometric Mixing Scale}
\label{sec:AMtoGM}
We briefly recall the construction from \cite{lunasin2012optimal} Section IV.B.
As a building block consider the ``hat'' function
\begin{align*}
  v(x)=
  \begin{cases}
    1-|x| ,&\text{ for } |x|\leq 1, \\
    0 ,&\text{ else}.
  \end{cases}
\end{align*}
Let $\epsilon=2^{-n}$ and let $\alpha \in \N$ with $\alpha< 2^{n-1}$. We then
built an odd function $u$ on $[-1,1]$ such that for $x>0$
\begin{align*}
  u(x)=  \alpha \epsilon v\left(\frac{x}{\alpha \epsilon}\right) + \sum_{j=1}^{2^{n}-2\alpha} \epsilon \phi\left(\frac{x-(2\alpha+2j) \epsilon}{\epsilon}\right).
\end{align*}
This function is a sawtooth function with one large tooth on the interval
$(0,2\alpha \epsilon)$ and smaller teeth of width $2\epsilon$ on the remainder
of $(0,1)$. Furthermore, $u$ is Lipschitz and $\rho=u' \in \{-1,1\}$ satisfies
\begin{align*}
  \|\rho\|_{H^{-1}}^{2}=\|u\|_{L^{2}}^{2}=\epsilon^{2}\left(\frac{1}{3}-\frac{2}{3}\alpha \epsilon\right)+ \frac{2}{3}\alpha^{3} \epsilon^{3} \approx \epsilon^{2}
\end{align*}
if $\epsilon$ is small. Taking $\alpha\approx \epsilon^{-1/3}$, we thus obtain
that $\|\rho\|_{H^{-1}}\approx \epsilon$ and that $\rho=1$ on $B_{\epsilon^{2/3}}(0)$
and is hence geometrically mixed at most at scale $\epsilon^{2/3}$.
\\

Additionally, we compute that if we consider larger radii $r \geq
\epsilon^{2/5}$, then
\begin{align*}
  \frac{1}{|B_{r}|} \left| \int_{B_{r}(\xi)} \rho dx dy \right| \leq C r
\end{align*}
This example hence shows that an estimate of the type $\epsilon \leadsto
\epsilon$ is not possible, but $\epsilon \leadsto \epsilon^\alpha$ is for this
case. In Sections \ref{sec:walsh} and \ref{sec:proofs} we show that such an
estimate indeed holds for general functions and in higher dimensions and that
our exponents of $\epsilon$ are optimal. Roughly speaking, the loss of power in
$\epsilon$ here is due to the $L^{1}$ normalization of the characteristic
functions and equivalence constants of (weighted) $l^{2}$ and $l^{\infty}$ norms
on finite-dimensional vector spaces (c.f. Section \ref{sec:walsh}) and agrees
with scaling (c.f. the proof of Theorem \ref{thm:main}).

\subsection{The Geometric Mixing Scale and Weak Limits}
\label{sec:modifiedGMintro}
Consider a periodic characteristic function $u\in L^{2}(\T)$ with $\int_{\T} u
dx =\frac{1}{2}$ and for $k \in \N$ define
\begin{align*}
  \rho_{k}(x)= \sgn(x) u(k|x|)  \in L^{2}((-1,1)).
\end{align*}
We note that $\int \rho_{k} dx =0, \|\rho_{k}\|_{L^{\infty}}=1$ and that
\begin{align*}
  \rho_{k} \xrightharpoonup{L^{2}} \frac{1}{2}\sgn(x).
\end{align*}
Hence, for any given ball $B_{r}(\xi) \subset (-1,1)$,
\begin{align*}
  \frac{1}{|B_{r}|} \left| \int_{B_{r}(\xi)} \rho dx \right| \rightarrow \frac{1}{|B_{r}|} \left| \int_{B_{r}(\xi)} \frac{1}{2}\sgn(x) dx \right| \leq \frac{1}{2}
\end{align*}
as $k \rightarrow \infty$. Using the periodicity to obtain more quantitative
estimates, we obtain that for any $\kappa > \frac{1}{2}$ and any $r>0$, we can
achieve
\begin{align*}
  \frac{1}{|B_{r}|} \left| \int_{B_{r}(\xi)} \rho_{k} dx \right| \leq \kappa \|\rho_{k}\|_{L^{\infty}} 
\end{align*}
provided $k$ is sufficiently large and that thus
$\mathcal{G}_{\kappa}[\rho_{k}]\rightarrow 0$ as $k \rightarrow \infty$.

However, this fails for $\kappa < \frac{1}{2}$ and by lower semicontinuity,
\begin{align*}
  \liminf_{k\rightarrow \infty} \|u_{k}\|_{H^{-1}} \geq \|\frac{1}{2}\sgn(x)\|_{H^{-1}} >0.
\end{align*}

Hence, this at first suggests that the geometric mixing scale is distinct from
the analytic mixing scale. In view of our comparison result, Theorem \ref{thm:mixing}, we instead
suggest to interpret this example as
\begin{align*}
  \|u_{k}\|_{H^{-1}} \leq C \max(\kappa, r)
\end{align*}
and note that the lower bound on $\kappa \geq \frac{1}{2}$ here is due to large
scale structures in the weak limit. Indeed, consider the functions $v_{k}$
obtained by projecting out large scales:
\begin{align*}
  v_{k}(x)=u_{k}(x)-\frac{1}{2}\sgn(x).
\end{align*}
Then it holds that
\begin{align*}
  \|v_{k}\|_{L^{\infty}} &\geq \kappa, \\
  \frac{1}{|B_{r}|} \left| \int_{B_{r}(\xi)} \rho_{k} dx \right| &\leq \min(2\kappa, C r) \text { for } r\geq \frac{c}{k}, \\
  \|v_{k}\|_{H^{-1}} &\leq \frac{C}{k}, \\
  v_{k}&\xrightharpoonup{L^{2}} 0 \text{ as } k \rightarrow \infty.
\end{align*}
\begin{rem}
  When considering $u(t_{k})=\rho(t_{k})$ for $\rho(t)$ evolving under ergodic
  dynamics, the weak limit is given by a constant function. In that setting, we
  may assume that constant to be zero after normalization and hence, there
  without loss of generality both notions of geometric mixing coincide, i.e.
  $v_{k}=u_{k}$.
\end{rem}
More generally, we don't need to know a candidate for the weak limit (or even
have a sequence), but rather consider the following setting: \\

If $\rho \in L^{2}\cap L^{\infty}\cap H^{-1}$ is such that for all $r\geq r_{0}$
and all $x_{0} \in \R^{n}$ it holds that
\begin{align*}
  \frac{1}{|B_{r}|} \left| \int_{B_{r}(x_{0})} \rho dx \right| \leq \kappa,
\end{align*}
then we claim that (c.f. Lemma \ref{lem:H1bygm})
\begin{align*}
  \rho_{r_{0}}= \rho - \left(\frac{1}{|B_{r_{0}}|}1_{B_{r_{0}}(\xi)} * \rho\right)
\end{align*}
satisfies
\begin{align*}
  \|\rho - \rho_{r_{0}}\|_{L^{\infty}} &\leq \kappa, \\
  \frac{1}{|B_{r}|} \left| \int_{B_{r}(\xi)} \rho_{r_{0}} dx \right| &\leq Cr, \\
  \|\rho_{r_{0}}\|_{H^{-1}} &\leq C r_{0}\|\rho\|_{L^{2}},
\end{align*}
for all $r \geq r_{0}$. Here, we stress that, while
\begin{align*}
  \frac{1}{|B_{r_{0}}|}1_{B_{r_{0}}(\xi)} * \rho \rightarrow \rho
\end{align*}
as $r_{0}\downarrow 0$ for fixed $\rho$, this is much more subtle for sequences
$\rho$ depending on $r_{0}$. Indeed, letting $u_{k}(x)$ be as above, we obtain
that
\begin{align*}
  \frac{1}{|B_{\frac{1}{k}}|} 1_{B_{\frac{1}{k}}}* u_{k}(x)= \frac{1}{2}\sgn(x)
\end{align*}
for all $x$ with $\text{dist}(x,\{0,\pi,-\pi\})\geq \frac{2}{k}$.

\section{A Walsh-Fourier Model}
\label{sec:walsh}

In order to introduce our ideas and establish sharpness of estimates, we first
discuss and compare both mixing scales in a dyadic model setting. Here, we
consider averages over dyadic intervals and replace the $\sin$ basis of
$L^{2}(\T)$ by functions that are constant $+1$ or $-1$ on dyadic intervals.
This setting is known in harmonic analysis as a Walsh-Fourier setting and
associated with a ``tile'' characterization and Haar wavelet expansions,
\cite{muscalu2004p}, \cite{thiele2000quartile}, \cite{thiele2000time}. In the
following we briefly provide some definitions and statements. For a more
in-depth introduction we refer the interested reader to \cite{thiele2006wave}.
We remark that, for simplicity of notation and estimates, we here consider the
setting of $L^{2}([0,1))$ instead of $L^{2}(\R)$. The dyadic setting has the
benefit of greatly simplifying estimates due to orthogonality and allows for
explicit computations of newly introduced analogues of the mixing scales as Besov-type norms in terms of certain
$L^{2}$ bases. Hence, here it is transparent what estimates are possible and
whether they are optimal. In Section \ref{sec:proofs} we show that, with minor
modifications, these results also extend to the continuous Sobolev setting.

\subsection{Definitions, Tiles and Bases}
\label{sec:definitions_dyadic}

\begin{defi}
  Let $[0,1)$ be the half-open unit interval. Then for each $j \in \N_{+}$, we
  define the set of \emph{dyadic intervals at scale $2^{-j}$} by
  \begin{align*}
    \mathcal{D}_{j}=\{I_{k,j}:=2^{-j}[k,k+1): k \in \{0, \dots, 2^{j}-1\}\}.
  \end{align*}
  Associated with this partition of $[0,1)$, we introduce the $L^{2}$-normalized
  characteristic functions
  \begin{align*}
    \chi_{I}= \frac{1}{\sqrt{|I|}}1_{I} \in L^{2}([0,1)).
  \end{align*}
\end{defi}
We note that, if $I, I' \in \mathcal{D}_{j}$, then either $I=I'$ or the
intervals are disjoint. If the intervals are not of the same size, that is $I
\in \mathcal{D}_{j}$ and $I' \in \mathcal{D}_{j'}$ with $j \neq j'$, they are
either disjoint or one is contained in the other.

In addition to the (normalized) characteristic functions, $\chi_{I}$, the
following definition introduces a large family of oscillating $L^{2}$ normalized
functions, which we use to define (fractional) Sobolev-type spaces.

\begin{figure}[htb]
  \includegraphics[width=0.5\linewidth]{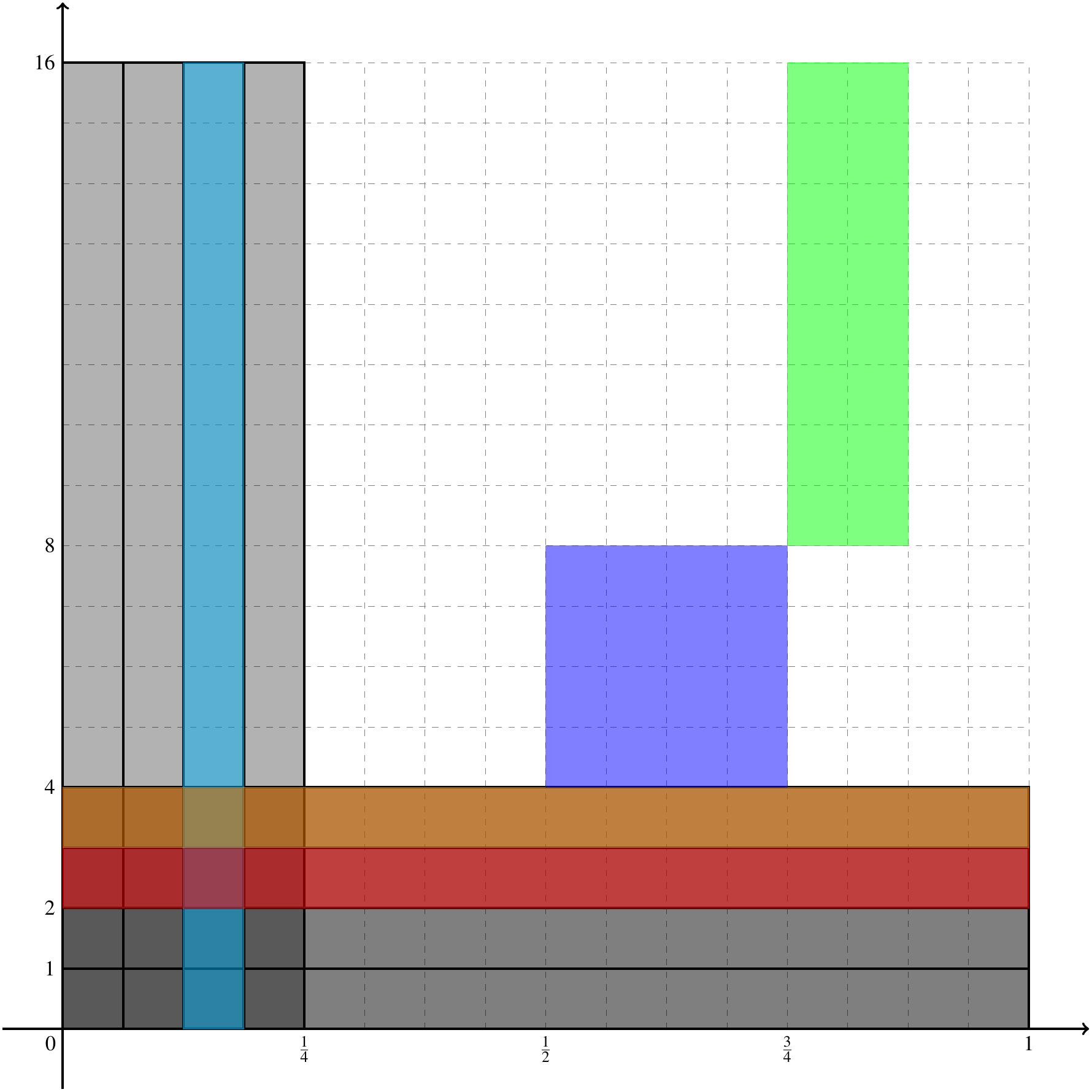}
  \label{fig:tiles}
  \caption{Various tiles down to scale $2^{-4}$. The vertical gray tiles
    correspond to characteristic functions $\chi_{I}$. The horizontal gray lines
    correspond to our replacement of a Fourier basis, $\phi_{p_{l}}$. The tiles
    in green and blue in the upper right corner are at level $l=1$. Plots of the
    corresponding wave packets of all colored tiles are given in Figure
    \ref{fig:wave_packets}. By Lemma \ref{lem:ortho} wave packs $\phi_{p},
    \phi_{p'}$ are $L^{2}$ orthogonal iff they are disjoint. }
\end{figure}

\begin{figure}[hbt]
  \includegraphics[width=0.3\linewidth]{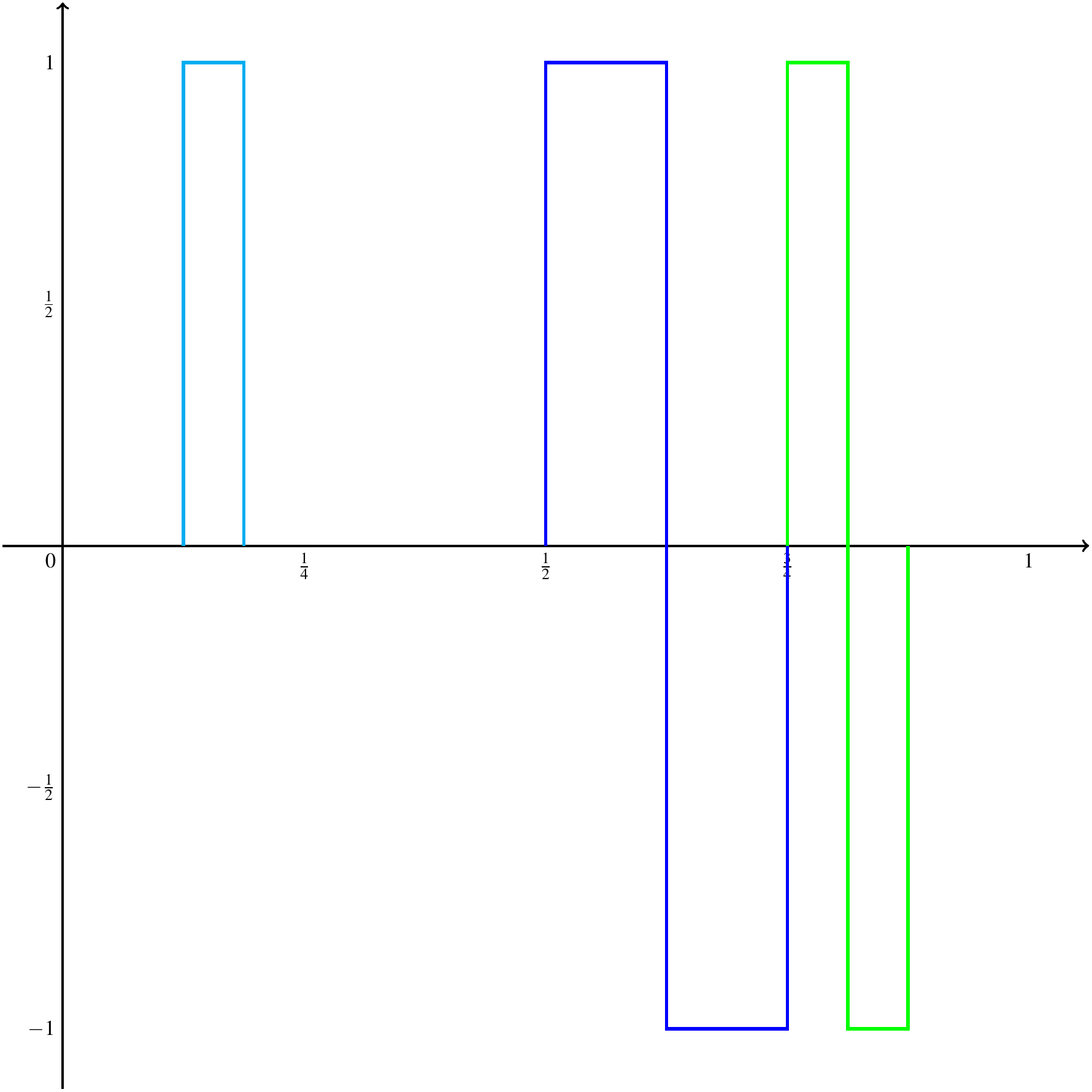}
  \includegraphics[width=0.3\linewidth]{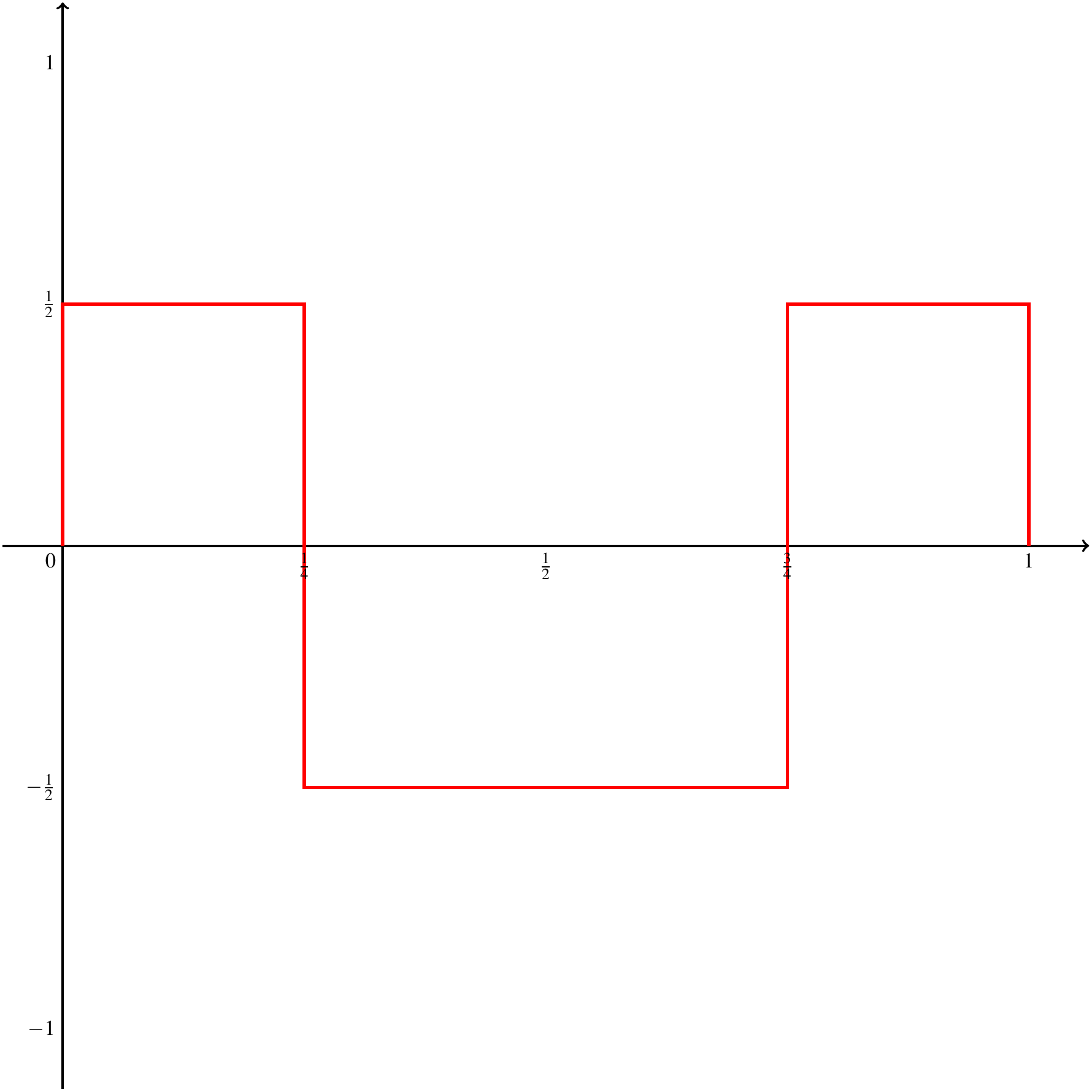}
  \includegraphics[width=0.3\linewidth]{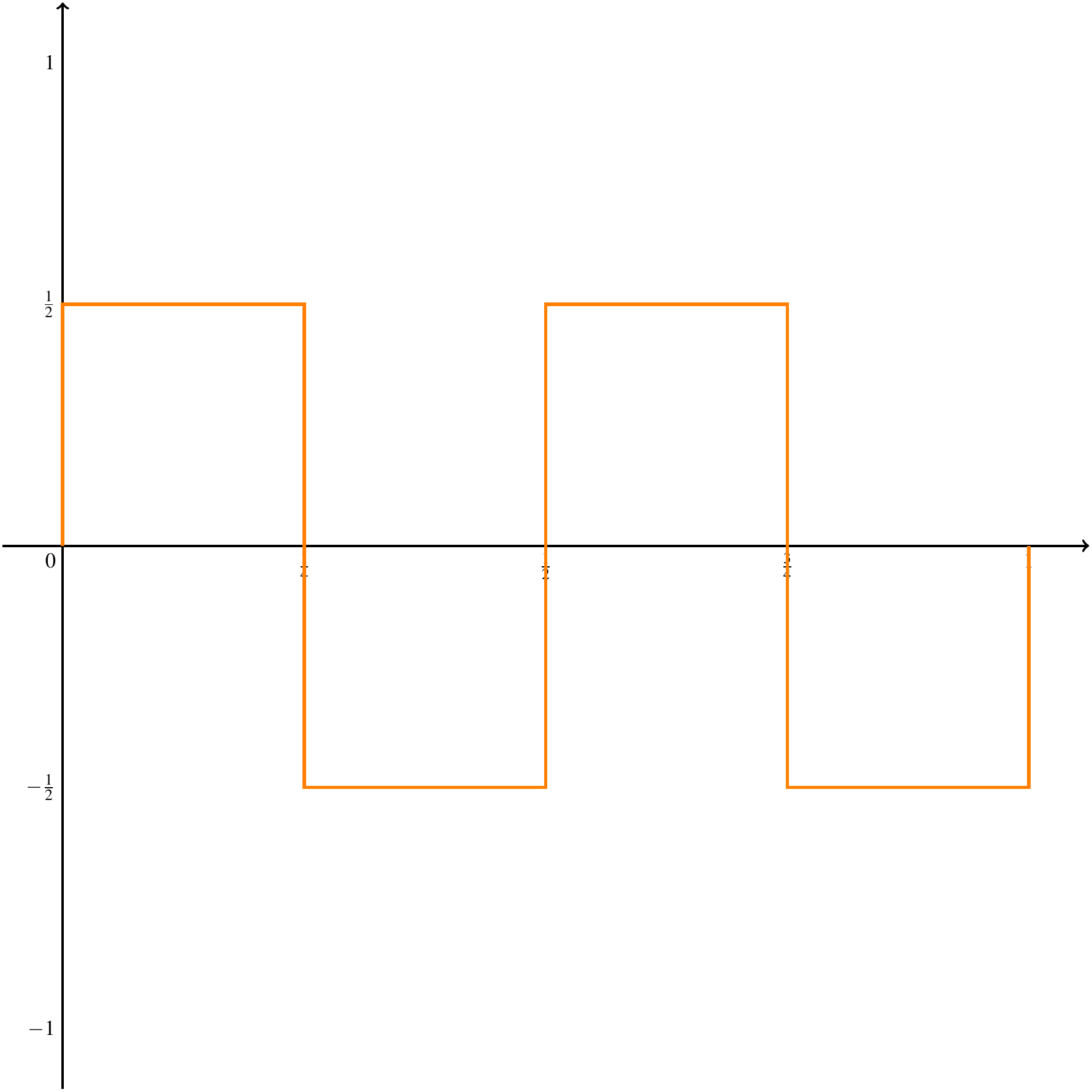}
  \caption{Walsh wave packets associated with the tiles of Figure
    \ref{fig:tiles}}
  \label{fig:wave_packets}
\end{figure}

\begin{defi}
  \label{defi:tiles}
  A \emph{tile} $p$ is a dyadic rectangle of area one in $[0,1) \times
  [0,\infty)$. That is,
  \begin{align*}
    p:= I \times \omega = [2^{-j}k,2^{-j}(k+1)) \times [2^{j}l, 2^{j}(l+1)),
  \end{align*}
  where $k \in \{0, 2^{j}-1\}, j \in \N_{0}, l \in \N_{0}$. If $l=0$, we define
  the \emph{wave-packet}
  $\phi_{p}:= \frac{1}{\sqrt{|I|}}1_{I}$. For $l>0$, we define 
  $\phi_{p}$ recursively. That is, if $p$ is a tile at level $l$ and scale
  $2^{j}$, we can express it as either upper or lower half the union of two
  tiles $p_{l}, p_{r}$ at level $\lfloor{\frac{l}{2}}\rfloor$ and scale
  $2^{-j-1}$. We thus define
  \begin{align*}
    \phi_{p}&= \frac{1}{\sqrt{2}} (\phi_{p_{l}}+\phi_{p_{r}}) \text{ if $l$ is even}, \\
    \phi_{p}&= \frac{1}{\sqrt{2}} (\phi_{p_{l}}-\phi_{p_{r}}) \text{ if $l$ is odd}.
  \end{align*}
\end{defi}

This definition allows us to consider questions of orthogonality and basis
expansions in a graphical way, a so-called \emph{cartoon} (c.f. Figure
\ref{fig:tiles}). The following lemma summarizes some of the main properties we
use in the following.

\begin{lem}[c.f. Lemma 2.9 in {\cite{thiele2000quartile}}]
  \label{lem:ortho}
  For any two tiles $p, p'$ we have
  \begin{align*}
    \left| \int_{[0,1)} \phi_{p} \phi_{p'} \right| = \sqrt{|p \cap p'|} 
  \end{align*}
  In particular two wave packets are $L^{2}$-orthogonal if and only if the underlying
  tiles have empty intersection.
\end{lem}

\begin{cor}[c.f. Corollary 2.7 in {\cite{thiele2000quartile}}]
  Furthermore, two families of tiles $P, P'$ cover the same region in $[0,1)
  \times [0,\infty)$ if and only if the spans of $\{\phi_{p}: p \in P\}$ and
  $\{\phi_{p}: p \in P'\}$ are identical. In particular, denoting $p_{l}=[0,1)
  \times [l, l+1)$, we obtain that $\{\chi_{I}\}_{I \in \mathcal{D}_{j}}$ and
  $\{\phi_{p_{l}}: l \in \{0,2^{j}-1\}\}$ are both orthonormal bases of the same
  space, which we denote by $E_{j}$. We further introduce the $L^{2}$ orthogonal
  \emph{projection operators} $P_{j}$ onto $E_{j}$.
\end{cor}

\begin{defi}[Mixing scales]
  \label{defi:dyadic_scales}
  Given $\rho \in L^{2}([0,1))$, we introduce the \emph{geometric mixing
    seminorms} at scale $j$ as
  \begin{align*}
    \mathfrak{g}_{j}[\rho]= \sup \left\{ \frac{1}{I} \left| \int_{I} \rho \right|: I \in \mathcal{D}_{j}\right\} = \sqrt{2^{j}} \sup \left\{ |\langle \chi_{I},\rho \rangle| : I \in \mathcal{D}_{j}  \right\}.
  \end{align*}

  Furthermore, for $s \in [-1,1]$ we define \emph{analytic mixing seminorms} up
  to scale $j$ by
  \begin{align*}
    \|\rho\|_{h^{s}_{j}}^{2} = \sum_{l=0}^{2^{j}-1} (1+l^{2})^{2} \left|\langle  \rho ,\phi_{p_{l}} \rangle  \right|^{2}.
  \end{align*}
  We further define $\|\rho\|_{h^{s}}= sup_{j} \|\rho\|_{h^{s}_{j}}$.
\end{defi}

\subsection{Estimates}
\label{sec:estimates}

We note that the geometric mixing seminorms can be expressed as
\begin{align*}
  \mathfrak{g}_{j}[\rho] = \sup_{I \in \mathcal{D}_{j}} \sqrt{2^{j}} |\langle \chi_{I}, \rho \rangle|
\end{align*}
and are hence indeed norms on the space $E_{j}$. Likewise the analytic seminorms
are weighted $l^{2}$ norms on the same space $E_{j}$ when expressed in the
orthonormal basis $\{\phi_{p_{l}}\}$. Hence, both mixing seminorms are
equivalent, with constants depending on $j$.

The following theorem establishes the corresponding estimates as well as related
estimates between different scales with uniform constants.
\begin{thm}
  \label{thm:dyadic_estimate}
  Let $\rho \in L^{2}((0,1))$, then for all $j \in \N$
  \begin{align*}
    \|\rho\|_{h^{-1}_{j}} &\leq \mathfrak{g}_{j}[\rho], \\
    \|\rho\|_{h^{-1}} &\leq \mathfrak{g}_{j}[\rho] + 2^{-j}\|(1-P_{j})\rho\|_{L^{2}},\\
    \mathfrak{g}_{j}[\rho] &\leq 2^{\frac{3}{2}j} \|\rho\|_{h^{-1}_{j}}.
  \end{align*}
  Furthermore, both seminorms depend on $\rho$ only via its projection, that is
  $\|\rho\|_{h^{-1}_{j}}=\|P_{j}\rho\|_{h^{-1}_{j}}$ and
  $\mathfrak{g}_{j}[\rho]=\mathfrak{g}_{j}[P_{j}\rho]$.
\end{thm}

We remark that the loss of the factor $2^{\frac{3}{2}j}$ here can be interpreted
as corresponding to the embedding $H^{-1}\subset H^{1/2}\subset L^{\infty}$.
Indeed, if $\rho \in E_{j}$, then $\rho$ is constant on each interval $I \in
\mathcal{D}_{j}$ and thus $\mathfrak{g}_{j}[\rho]=\|\rho\|_{L^{\infty}}$.

\begin{cor}
  \label{cor:dyadic_estimate}
  Let $\rho \in L^{2}((0,1))$, then
  \begin{enumerate}
  \item If $\|\rho\|_{h^{-1}_{j_{0}}}\leq 2^{-j_{0}}$, then
    $\mathfrak{g}_{j}[\rho] \leq C 2^{\frac{3}{2}j-j_{0}}$. In particular, if $j
    \leq \frac{2}{3}j_{0}$, then $\rho$ is geometrically mixed at scale
    $2^{-j}$.
  \item Furthermore, for $j \leq \frac{2}{5}j_{0}$, we obtain that
    $\mathfrak{g}_{j}[\rho] \leq 2^{-j}$.
  \item Conversely, if $\mathfrak{g}_{j}[\rho]\leq 2^{-j}$, then
    \begin{align*}
      \|\rho\|_{h^{-1}_{j}} \leq \mathfrak{g}_{j}[\rho] \leq 2^{-j}.
    \end{align*}
    If we additionally assume that $\|(1-P_{j})\rho\|_{L^{2}}\leq 1$, then
    furthermore
    \begin{align*}
      \|\rho\|_{h^{-1}} \leq \mathfrak{g}_{j}[\rho] \leq 2^{-j} + 2^{-j}\|(1-P_{j})\rho\|_{L^{2}} \leq C 2^{-j}.
    \end{align*}
  \end{enumerate}
\end{cor}

\begin{rem}    
  Denoting $2^{-j_{0}}=\epsilon$, the above results show that
  $\|\rho\|_{h^{-1}}\leq \epsilon$ implies
  $\mathfrak{g}_{\log(\epsilon^{2/3})}[\rho] \leq C$ and
  $\mathfrak{g}_{\log(\epsilon')}[\rho]\leq \epsilon'$ for $\epsilon'\geq
  \epsilon^{2/5}$.
  
  Conversely, if $\mathfrak{g}_{\log(\epsilon)}[\rho] \leq \epsilon$ and we
  control $\|\rho\|_{L^{2}}$, then also $\|\rho\|_{h^{-1}}\leq \epsilon$.

  The analytic and geometric mixing scales are hence comparable with a loss in
  the exponent in one direction (we hence do not use the word equivalent).
  
  Furthermore, we show in Lemma \ref{lem:dyadic_optimal} that this loss is
  optimal.

  As mentioned in the introductory Section \ref{sec:examples}, in this article
  we hence stress the viewpoint that the examples constructed in
  \cite{lunasin2012optimal} should instead be interpreted as showing the
  necessity of the control of $\mathfrak{g}_{\log(\epsilon)}[\rho] \leq
  \epsilon$ instead of $\mathfrak{g}_{\log(\epsilon)}[\rho] \leq \kappa$ and of the loss $\epsilon \rightarrow
  \epsilon^{2/3}$.

  We discuss this interpretation, scaling and the constructions further in
  Section \ref{sec:proofs}.
\end{rem}

\begin{proof}[Proof of Theorem \ref{thm:dyadic_estimate}]
  Let $\rho \in L^{2}$ be given and consider the basis expansions of $P_{j}\rho
  \in E_{j}$:
  \begin{align*}
    P_{j}\rho= \sum_{I} d_{I} \chi_{I}= \sum_{l} c_{l} \phi_{p_{l}}.  
  \end{align*}
  Since both $\chi_{I}$ and $\phi_{p_{l}}$ are orthonormal bases of $E_{j}$, it
  follows that
  \begin{align*}
    \|d_{I}\|_{l^{2}}=\|c_{l}\|_{l^{2}}.
  \end{align*}
  Then we may estimate
  \begin{align*}
    \|\rho\|_{h^{-1}_{j}} = \|<l>^{-1}c_{l}\|_{l^{2}} \leq \|c_{l}\|_{l^{2}}= \|d_{I}\|_{l^{2}} \leq \sqrt{2^{j}} \max{|d_{I}|}.
  \end{align*}
  On the other hand, the normalization of the geometric mixing functionals is such that
  \begin{align*}
    \mathfrak{g}_{j}[\rho]= \max \frac{1}{|I|} \langle \chi_{I},\rho \rangle = \sqrt{2^{j}}\max{|d_{I}|}.
  \end{align*}
  For the converse estimate, we note that
  \begin{align*}
    \mathfrak{g}_{j}[\rho]&= \sqrt{2^{j}}\max{|d_{I}|} \leq \sqrt{2^{j}} \|d_{I}\|_{l^{2}}\\
                          &= \sqrt{2^{j}} \|c_{l}\|_{l^{2}} \leq \sqrt{2^{j}}2^{j} \|<l>^{-1}c_{l}\|_{l^{2}} \\
                          &\leq 2^{\frac{3}{2}j}\|\rho\|_{h^{-1}_{j}}.
  \end{align*}

\noindent
If $\rho \not \in E_{j}$, we note that by definition
\begin{align*}
  \|\rho\|_{h^{-1}_{j}}^{2}= \sum_{l=0}^{2^{j}-1} \frac{1}{1+l^{2}} |\langle \rho, \phi_{p_{l}} \rangle|^{2} \leq \sum_{l=0}^{\infty} \frac{1}{1+l^{2}} |\langle \rho, \phi_{p_{l}} \rangle|^{2}= \|\rho\|_{h^{-1}}^{2}.
\end{align*}
For the estimate of $\|\rho\|_{h^{-1}}$, we thus split $\rho$ using $P_{j},
1-P_{j}$ and obtain
\begin{align*}
  \|\rho\|_{h^{-1}}^{2} &= \|\rho\|_{h^{-1}_{j}}^{2} + \sum_{l\geq 2^{j}} \frac{1}{1+l^{2}} |\langle \rho, \phi_{p_{l}} \rangle|^{2}
  \\ &\leq \mathfrak{g}_{j}[\rho]^{2} + 2^{-2j}\|\rho\|_{L^{2}}^{2}. 
\end{align*}
\end{proof}

\begin{proof}[Proof of Corollary \ref{cor:dyadic_estimate}]
  We apply Theorem \ref{thm:dyadic_estimate} to obtain that
  \begin{align*}
    \mathfrak{g}_{j}[\rho] \leq C 2^{\frac{3}{2}j-j_{0}}.
  \end{align*}
  The first statements hence follow by noting that
  \begin{align*}
    \frac{3}{2}j -j_{0}\leq 0 \Leftrightarrow j \leq \frac{2}{3}j_{0},
  \end{align*}
  and the second from
  \begin{align*}
    \frac{3}{2}j -j_{0}\leq -j \Leftrightarrow j \leq \frac{2}{5}j_{0}.
  \end{align*}
  The last statement similarly follows as a direct corollary of Theorem
  \ref{thm:dyadic_estimate}.
\end{proof}

The following lemma shows that these restrictions on $j$ are optimal.
\begin{lem}
  \label{lem:dyadic_optimal}
  There exists a family of functions $\rho=\rho(j_{0})$ such that
  $\|\rho\|_{h^{-1}_{j_{0}}}\leq C 2^{-j_{0}}$ and which satisfies the following
  properties:
  \begin{enumerate}
  \item For any $\alpha < \frac{2}{3}$ and $j=\lfloor {\alpha j_{0}} \rfloor$,
    it holds that $\mathfrak{g}_{j}[\rho] =o(1)$ as $j_{0} \rightarrow \infty$.
    If instead $\alpha > \frac{2}{3}$, then $\mathfrak{g}_{j}[\rho]\rightarrow
    \infty$.
  \item For any $\alpha < \frac{2}{5}$ and $j=\lfloor {\alpha j_{0}} \rfloor$,
    $2^{j} \mathfrak{g}_{j}[\rho] = o(1)$ as $j_{0}\rightarrow \infty$. If
    instead $\alpha > \frac{2}{5}$, then $2^{j}
    \mathfrak{g}_{j}[\rho]\rightarrow \infty$.
  \end{enumerate}
\end{lem}

\begin{proof}
  As shown in the preceding Theorem \ref{thm:dyadic_estimate} and Corollary
  \ref{cor:dyadic_estimate}, we have
  \begin{align*}
    \mathfrak{g}_{j}[\rho]= \sqrt{2^{j}}\|d_{I}\|_{l^{\infty}} \leq \sqrt{2^{j}}\|d_{I}\|_{l^{2}} \leq 2^{\frac{3}{2}j}\|\rho\|_{H^{-1}} \leq 2^{\frac{3}{2}j-j_{0}}.
  \end{align*}
  Hence, we note that $\frac{3}{2}j-j_{0}\leq 0 \Leftrightarrow j \leq
  \frac{2}{3}j_{0}$ and that $\frac{3}{2}j -j_{0}\leq j \Leftrightarrow j \leq
  \frac{2}{5}$.

  It hence only remains to show that these estimates are indeed sharp in the
  thresholds in $j$. For this purpose, consider
  \begin{align*}
    \rho=\sum_{2^{\alpha j_{0}}-1 \leq l \leq 2^{j_{0}}-1} \phi_{p_{l}}. 
  \end{align*}
  with $\alpha \in (0,1)$ to be chosen later (c.f. also Remark \ref{rem:onalpha}). Then we can compute
  \begin{align*}
    \|\rho\|_{h^{-1}_{j_{0}}} = \left( \sum_{2^{\alpha j_{0}}-1 \leq l \leq 2^{j_{0}}-1} l^{-2}  \right)^{\frac{1}{2}}\approx 2^{-j_{0} \frac{\alpha}{2}}.
  \end{align*}
  On the other hand, averaging over the interval $[0,2^{-j})$, wave packets
  $\phi_{p_{l}}$ with $l>2^{-j}$ are orthogonal, while wave packets with $l\leq
  2^{j}$ are constantly equal to one when restricted to this interval. Hence, we
  obtain that for any $j \leq j_{0}$
  \begin{align*}
    \mathfrak{g}_{j}[\rho] = \sum_{2^{j_{0}\alpha}-1 \leq l \leq 2^{j-1}} 1 \approx 2^{j}
  \end{align*}
  if $j> j_{0} \alpha$. If we now multiply $\rho$ by
  $2^{-j_{0}(1-\frac{\alpha}{2})}$, we are exactly in the setting described.
  That is,
  \begin{align*}
    \|2^{-j_{0}(1-\frac{\alpha}{2})} \rho\|_{H^{-1}} \approx 2^{-j_{0}(1-\frac{\alpha}{2})}2^{-j_{0} \frac{\alpha}{2}} =2^{-j_{0}}
  \end{align*}
  and
  \begin{align*}
    \mathfrak{g}_{j}[2^{-j_{0}(1-\frac{\alpha}{2})} \rho] \approx 2^{j-j_{0}(1-\frac{\alpha}{2})}
  \end{align*}
  for any $j$ with $\alpha j_{0} \leq j \leq j_{0}$. Since the exponent is
  monotone in $j$, we only need to consider the case when $\alpha j_{0}=j$ and
  thus the behavior of $(\alpha-(1-\frac{\alpha}{2}))$. This exponent is less or
  equal than zero if and only if $\alpha \leq \frac{2}{3}$ and less or equal
  than $-\alpha$ if and only if $\alpha \leq \frac{2}{5}$. The thresholds in
  $j$, respectively $\alpha$, are thus indeed optimal.
\end{proof}
\begin{rem}
  \label{rem:onalpha}
  We remark that $\langle \chi_{[0,2^{-j})}, \phi_{p_{l}} \rangle = 2^{-j}$ for
  all $j=0,\dots, 2^{j-1}$. Hence, $\sum_{0 \leq l \leq 2^{j}-1} \phi_{p_{l}}=
  2^{j}1_{[0,2^{-j})}$. Thus, if $\alpha j_{0}$ is an integer, we obtain that
  \begin{align*}
    \sum_{2^{ \alpha j_{0}} \leq l \leq
    2^{j_{0}}-1} \phi_{p_{l}} = 2^{j_{0}}1_{[0,2^{-j_{0}})} -  2^{\alpha j_{0}}1_{[0,2^{-\alpha j_{0}})},
  \end{align*}
  which provides a more immediate view of the geometric mixing size. However,
  this explicit characterization is much less simple if $2^{\alpha j_{0}}$ is
  not a power of two and also intransparent in terms of the $H^{-1}$ norm.
\end{rem}

\section{The Continuous Setting}
\label{sec:proofs}
In the following we show that, with minor modifications, the estimates of the
dyadic setting of Section \ref{sec:walsh} persist in the continuous setting.
Here, additional key challenges are given by the lack of orthogonality and thus
non-existence of spaces like $E_{j}$.
\subsection{Definition of Mixing Scales}
\label{sec:definitions}
\begin{defi}
  \label{defi:scales}
  If $\rho \in \dot H^{-1}(\R^{n})$, we call $\|\rho\|_{\dot H^{-1}}$ the \emph{analytic
    mixing scale}.
  \\
  
  Let $\phi \in L^{1}(\R^{n})$ with $\phi\geq 0$ and $\|\phi\|_{L^{1}}=1$ and denote
  $\phi_{r}(x):=\frac{\phi(rx)}{r^{n}}$. Then for any $\rho \in
  L^{1}_{\loc}(\R^{n})$ and every $\epsilon_{0}>0$, we introduce the (nonlinear)
  functionals
  \begin{align*}
    \mathfrak{g}_{\epsilon_{0}}[\rho] := \|\phi_{r} * \rho\|_{L^{\infty}}.
  \end{align*}
  Here, the most common choice is given by $\phi=\frac{1}{|B_{1}|}1_{B_{1}}$, in
  which case
  \begin{align*}
    \mathfrak{g}_{\epsilon_{0}}[\rho] = \sup_{r> \epsilon_{0}, x \in \R^{n}} |B_r(x)|^{-1} \left| \int_{B_r(x)} \rho(y) dy \right|.
  \end{align*}

  We say that a function $\rho \in L^\infty(\Omega) \cap L^1_{loc}(\Omega)$ is
  \emph{geometrically mixed} by a factor $\kappa \in (0,1)$ up to scale
  $\epsilon_0>0$ if
  \begin{align*}
    \mathfrak{g}_{\epsilon_{0}}[\rho] \leq \kappa \|\rho\|_{L^{\infty}}.
  \end{align*}
  For a given $\kappa$, we denote
  \begin{align*}
    \mathcal{G}_{\kappa}[\rho]= \inf \{\epsilon_{0}>0: \mathfrak{g}_{\epsilon_{0}}[\rho] \leq \kappa \|\rho\|_{L^{\infty}}\},
  \end{align*}
  the infimum over all such $\epsilon_0$ and call it the \emph{geometric mixing
    scale}.
\end{defi}

We remark that, by Hölder's inequality,
\begin{align*}
  \mathfrak{g}_{\epsilon_{0}}[\rho] \leq \|\rho\|_{L^{\infty}} \|\phi_{r}\|_{L^{1}} =\|\rho\|_{L^{\infty}},
\end{align*}
and that by Lebesgue integration theory
\begin{align*}
  \lim_{\epsilon_{0} \downarrow 0}\mathfrak{g}_{\epsilon_{0}}[\rho] = \|\rho\|_{L^{\infty}}.
\end{align*}
The functionals $\mathfrak{g}$ and the geometric mixing scale $\mathcal{G}$ thus
describe the competition between cancellations in Hölder's inequality and
convergence of Dirac sequences.

\begin{rem}
  \label{rem:regularity_assumptions}
  The reason for our more general formulation in terms of $\phi \in L^{1}$ is
  that in later estimates optimality is easier to phrase and establish if we
  additionally require that $\phi \in H^{1}$. In particular, by duality
  \begin{align*}
    \sup_{\rho \in H^{-s}: \|\rho\|_{H^{-s}}\leq 1}
    \mathfrak{g}_{1}[\rho]= \|\phi\|_{H^{s}}
  \end{align*}
  and hence an estimate like \eqref{eq:4} is not possible unless $\phi$ is
  sufficiently regular. However, for most estimates this only poses technical
  challenges (c.f. Lemma \ref{lem:withFouriersupport}) in terms of the control
  of certain Fourier projections. In the dyadic setting of Section
  \ref{sec:walsh} these complications could be avoided by using orthogonality
  properties.
\end{rem}

\subsection{Comparison Estimates}

Our main results are given by the following theorems and corollaries.
\begin{thm}[Estimates]
  \label{thm:main}
  Let $\rho \in L^{2}\cap \dot H^{-1}(\R^{n})$ and suppose that $\phi \in H^{
    \lambda}\cap L^{1}$ for some $\lambda \in [0,1]$. Then for any
  $\epsilon_{0}>0$ it holds that
  \begin{align*}
    \mathfrak{g}_{\epsilon_{0}}[\rho] &\leq C \frac{\|\phi\|_{H^{\lambda}}}{\|\phi\|_{L^{1}}} \epsilon_{0}^{-n/2-\lambda}\|\rho\|_{H^{-\lambda}}.
  \end{align*}
  As a consequence, the geometric mixing scale of $\rho$ can be estimated by
  \begin{align*}
    \mathcal{G}_{\kappa}(\rho) & \leq  \left( \frac{C_{\lambda, \phi}\|\rho\|_{H^{-\lambda}}}
                                 {\kappa \|\rho\|_{L^{\infty}}}\right)^{1/(n/2+\lambda)}.
  \end{align*}
  In particular, if $\lambda=1$, then
  \begin{align}
    \label{eq:4}
    \begin{split}
      \mathfrak{g}_{\epsilon_{0}}[\rho] &\leq C \epsilon_{0}^{-n/2-1} \|\rho\|_{H^{-1}}, \\
      \mathcal{G}_{\kappa}(\rho) &\leq C \left( \frac{\|\rho\|_{H^{-1}}}{\kappa
          \|\rho\|_{L^{\infty}}} \right)^{\frac{1}{n/2+1}}.
    \end{split}
  \end{align}

  Conversely, if $\rho$ is compactly supported and $C$ denotes the measure of a
  $1$-neighborhood of the support, then for every $\epsilon_{0} \leq 1$ it holds
  that
  \begin{align}
    \label{eq:5}
    \|\rho\|_{H^{-1}} \leq C \mathfrak{g}_{\epsilon_{0}}[\rho] + C \epsilon_{0}\|\rho\|_{L^{2}}.
  \end{align}
\end{thm}

We note that in the estimate \eqref{eq:5}, assuming
$\mathcal{G}_{\kappa}[\rho]\leq \epsilon_{0}$ only yields a bound of
$\|\rho\|_{H^{-1}}\leq \kappa$. Indeed, as explored in Section
\ref{sec:modifiedGMintro} for fixed $\kappa$ it is possible to find a sequence
$\rho_{n}$ such that
\begin{align*}
  \mathcal{G}_{\kappa}[\rho_{n}] \rightarrow 0, \\
  \|\rho_{n}\|_{H^{-1}} \geq \kappa,
\end{align*}
where the failure of decay of $\|\rho_{n}\|_{H^{-1}}$ was due to the persistence
of structures at scale $\kappa$ (c.f. also Theorem \ref{thm:continuous} and the
remarks thereafter).

As discussed in Remark \ref{rem:regularity_assumptions}, if $\phi=c 1_{B_1}$, we
can not choose $\lambda=1$ since $\phi \not \in H^{1}$. However, we may recover
this estimate upon imposing further conditions on $\rho$ (c.f. Lemma
\ref{lem:withFouriersupport}).

As an application of the above estimates we derive comparability of both mixing
scales. That is, while not equivalent (semi-)norms, smallness of one scale
implies smallness of the other with a necessary loss in the exponents.
For easier reference, we restate Theorem \ref{thm:continuous}
\begin{thm}[Comparison of mixing scales]
  Let $\rho \in L^{2}(\R^{n})$ and $\|\rho\|_{L^{2}}\leq 1$ and let $\phi \in
  H^{1}$. Then for all $0<\epsilon \leq 1$ it holds that:
  \begin{enumerate}
  \item If $\mathfrak{g}_{\epsilon}[\rho]\leq \epsilon$ and $\rho$ is supported
    in $B_{1}$, then also $\|\rho\|_{H^{-1}}\leq C \epsilon$.
  \item If $\|\rho\|_{H^{-1}}\leq \epsilon$, then also
    $\mathfrak{g}_{\tilde{\epsilon}}[\rho] \leq C$ for all $\tilde{\epsilon}\geq
    \epsilon^{\alpha}$ and $\mathfrak{g}_{\epsilon'}[\rho]\leq C \epsilon'$ for
    all $\epsilon'\geq \epsilon^{\beta}$, where $\alpha=\frac{2}{n+2}$ and
    $\beta=\frac{2}{n+4}$ depends only on the dimension. In particular,
    supposing additionally that $\|\rho\|_{L^{\infty}}=1$, it follows that
    \begin{align*}
      \mathcal{G}_{C}[\rho]&\leq \tilde{\epsilon}, \\
      \mathcal{G}_{C\epsilon'}[\rho]&\leq \epsilon',
    \end{align*}
  \end{enumerate}
  These estimates are optimal in the powers of $\epsilon$.
\end{thm}

In Section \ref{sec:walsh} we have seen that the loss of exponents is caused by the
$(L^{1}, L^{\infty})$ normalization in the geometric scale instead of $L^{2}$ normalization
for the analytic scale and can also be seen as being due the Sobolev embedding into
$L^{\infty}$. In this continuous setting, this is much less transparent due
to the lack of spaces $E_{j}$. The necessity of the loss is established in Lemma
\ref{lem:continuous_optimal} in analogy with Lemma \ref{lem:dyadic_optimal} of
the dyadic case.
\begin{cor}
  Let $u:\R_{+}\times \R^{n}\rightarrow \R$ be such that
  \begin{align*}
    \|\rho(t)\|_{H^{-1}} \geq C e^{-Ct}.
  \end{align*}
  Then, it follows that
  \begin{align*}
    \mathfrak{g}_{e^{-Ct}}[\rho(t)] \geq C e^{-Ct}.
  \end{align*}
\end{cor}

\begin{rem}
  For example $u(t)$ may be given by the solution of a passive or active scalar
  problem, as in \cite{Crippa17}. As noted in Section \ref{sec:examples}, if one
  were instead to consider $\mathfrak{g}_{\kappa}[\rho(t)]$ for fixed $\kappa$,
  this functional is not lower semi-continuous and there is no reason to expect any lower
  bound.
\end{rem}

\begin{proof}[Proof of Theorem \ref{thm:main}]
  As a first consistency check, we verify that the estimates of Theorem
  \ref{thm:main} scale correctly.

  Let thus $\delta>0$ and consider $\rho_{\delta}(x)=\rho(\delta x)$. For
  simplicity of notation, we consider $\lambda=\frac{1}{2}$. Then it holds that
  \begin{align*}
    \mathfrak{g}_{\epsilon_{0}}[\rho_{\delta}]&= \mathfrak{g}_{ \delta \epsilon_{0}}[\rho] \leq C (\delta \epsilon_{0})^{-n/2-1/4} \|\rho\|_{L^{2}}^{\frac{1}{2}}\|\rho\|_{H^{-1}}^{\frac{1}{2}}\\
                                              &= C  \delta^{-n/2-1/4} \epsilon_{0}^{-n/2-1/4} \|\rho\|_{L^{2}}^{\frac{1}{2}}\|\rho\|_{H^{-1}}^{\frac{1}{2}}.
  \end{align*}
  On the other hand, estimating directly, we obtain
  \begin{align*}
    \mathfrak{g}_{\epsilon_{0}}[\rho_{\delta}] &\leq C \epsilon_{0}^{-n/2-1/4} \|\rho_{\delta}\|_{L^{2}}^{\frac{1}{2}}\|\rho_{\delta}\|_{H^{-1}}^{\frac{1}{2}}\\
                                               &= C \epsilon_{0}^{-n/2-1/4} (\delta^{-n/2}\|\rho\|_{L^{2}})^{1/2} (\delta^{-n/2-1/2}\|\rho\|_{L^{2}})^{1/2}\\
                                               &= C  \delta^{-n/2-1/4} \epsilon_{0}^{-n/2-1/4} \|\rho\|_{L^{2}}^{\frac{1}{2}}\|\rho\|_{H^{-1}}^{\frac{1}{2}}.
  \end{align*}
  Note that this estimate is further invariant under replacing $\rho(x)$ by $\mu
  \rho(x)$ for any $\mu>0$. Hence, we may in addition choose $\mu=\delta^{n/2}$
  to ensure $L^{2}$ normalization.
  \\

  Let us now consider the proof of the theorem and let $\lambda \in [0,1]$ and
  $\phi \in H^{\lambda}$ be given. Then we may use duality to estimate
  \begin{align*}
    \mathfrak{g}_{r}[\rho] \leq \frac{\|\phi_{r}\|_{H^{\lambda}}}{\|\phi_{r}\|_{L^{1}}} \|\rho\|_{H^{-\lambda}}
  \end{align*}
  and use interpolation to control
  \begin{align*}
    \|\rho\|_{H^{-\lambda}} \leq C_{\lambda} \|\rho\|_{L^{2}}^{1-\lambda}\|\rho\|_{H^{-1}}^{\lambda}.
  \end{align*}
  We further note that by scaling
  \begin{align*}
    \frac{\|\phi_{r}\|_{H^{\lambda}}}{\|\phi_{r}\|_{L^{1}}}  \leq C r^{-n/2-\lambda} \frac{\|\phi\|_{H^{\lambda}}}{\|\phi\|_{L^{1}}}
  \end{align*}
  for $r\leq 1$ (for the homogeneous Sobolev norms we would have equality).
 
  Combining both estimates, we obtain
  \begin{align*}
    \mathfrak{g}_{r}[\rho] \leq C\frac{\|\phi\|_{H^{\lambda}}}{\|\phi\|_{L^{1}}} r^{-n/2-\lambda}\|\rho\|_{L^{2}}^{1-\lambda}\|\rho\|_{H^{-1}}^{\lambda}.
  \end{align*}

  For the estimate on the geometric mixing scale, we have to show that for given
  $\kappa$ and all $\epsilon_{0} \geq \mathcal{G}[\rho]$
  \begin{align*}
    \mathfrak{g}_{\epsilon_{0}}[\rho] \leq  \kappa \|\rho\|_{L^{\infty}}.
  \end{align*}
  In view of the previous calculation this is implied by showing that
  \begin{align*}
    C_{\lambda,\phi}\epsilon_{0}^{-n/2-\lambda/2}\|\rho\|_{L^{2}}^{1-\lambda}\|\rho\|_{H^{-1}}^{\lambda} \leq \kappa \|\rho\|_{L^{\infty}},
  \end{align*}
  with $C_{\lambda,\phi}=C(\frac{\|\phi\|_{H^{\lambda}}}{\|\phi\|_{L^{1}}})$.
  Dividing by $\kappa \|\rho\|_{L^{\infty}}>0$ and taking a power
  $1/(n/2+\lambda/2)$, we obtain that this holds if
  \begin{align*}
    \left( \frac{C_{\lambda}\|\rho\|_{L^{2}}^{1-\lambda}\|\rho\|_{H^{-1}}^{\lambda}}{\kappa \|\rho\|_{L^{\infty}}}\right)^{1/(n/2+\lambda/2)} \leq \epsilon_{0}.
  \end{align*}
  We thus obtain an upper bound on the geometric mixing scale by the
  left-hand-side, which concludes the proof.
\end{proof}

\begin{proof}[Proof of Theorem \ref{thm:continuous}]
  We proceed as in the proof of Corollary \ref{cor:dyadic_estimate} and consider
  $\epsilon_{0}=\epsilon^{t}$ for $t \in (0,1)$ to be determined. Then the
  estimate \eqref{eq:4} of Theorem \ref{thm:main} implies that
  \begin{align*}
    \mathfrak{g}_{\epsilon^{t}} \leq C \epsilon^{t(-n/2-1)+1},
  \end{align*}
  which yields the critical cases
  \begin{align*}
    t(-n/2-1)+1=0 \Leftrightarrow t = \frac{2}{n+2} \\
    t(-n/2-1)+1=t \Leftrightarrow t = \frac{2}{n+4}
  \end{align*}
\end{proof}

The following lemmata consider questions of optimality and the removal of small
scales discussed in Section \ref{sec:modifiedGMintro}.

\begin{lem}[Counter example in the continuous setting]
  \label{lem:continuous_optimal}
  There exists a sequence $\epsilon \downarrow 0$ and $\rho=\rho(\epsilon) \in
  L^{2}(\R)$ with
  \begin{align*}
    \|\rho\|_{H^{-1}}\leq \epsilon,
  \end{align*}
  but such that for every $\alpha<\frac{2}{3}$, it holds that
  \begin{align*}
    \mathfrak{g}_{\epsilon^{\alpha}}[\rho] \rightarrow \infty. 
  \end{align*}
  as $\epsilon \downarrow 0$ and such that for all $\beta< \frac{2}{5}$
  \begin{align*}
    \epsilon^{-\beta} \mathfrak{g}_{\epsilon^{\beta}}[\rho] \rightarrow \infty.
  \end{align*}
  That is, the exponents in Theorem \ref{thm:continuous} are optimal.
\end{lem}

\begin{proof}[Proof of Lemma \ref{lem:continuous_optimal}]
  We follow a similar strategy as in the proof of Lemma \ref{lem:dyadic_optimal}
  in the dyadic setting. Let $\epsilon=2^{-j_{0}}$ and consider
  \begin{align*}
    \rho=2^{j_{0}}1_{[0,2^{-j_{0}}]}- 2^{j_{1}}1_{[0,2^{-j_{1}}]},
  \end{align*}
  with $j_{1}=\alpha j_{0}$, $\alpha \in (0,1)$. Then for any $j_{1}\leq j \leq
  j_{0}$, we obtain
  \begin{align*}
    2^{j}\int_{0}^{2^{-j}}\rho dx = 2^{j} (2^{j_{0}} \min(2^{-j}, 2^{-j_{0}}) + 2^{j_{1}} \min(2^{-j}, 2^{-j_{1}})) = 2^{j} - 2^{j_{1}},
  \end{align*}
  which is comparable to $2^{j}$ as long as $j > j_{1}$.

  We further make the following \underline{claim:}
  \begin{align}
    \|\rho\|_{H^{-1}} \approx 2^{-j_{1}/2}.
  \end{align}
  Suppose that this claim holds and consider
  \begin{align*}
    \rho':=2^{-j_{0}(1-\alpha/2)}\rho,
  \end{align*}
  which satisfies
  \begin{align*}
    \|\rho'\|_{H^{-1}} &\approx 2^{-j_{0}}, \\
    \mathfrak{g}_{2^{-j}}[\rho'] &\approx 2^{j-j_{0}(1-\alpha/2)}= 2^{-j_{0}(1-\frac{3}{2}\alpha)}.
  \end{align*}
  We hence conclude as in the proof of Lemma \ref{lem:dyadic_optimal}.

  It hence remains to show the claim. We directly compute
  \begin{align*}
    \hat{\rho}(\xi)= \int_{\R} e^{i\xi x} 2^{j_{0}}1_{[0,2^{-j_{0}}]}- 2^{j^{1}}1_{[0,2^{-j_{1}}]} dx = \frac{e^{i\xi 2^{-j_{0}}}-1}{i\xi 2^{-j_{0}}} - \frac{e^{i\xi 2^{-j_{1}}}-1}{i\xi 2^{-j_{1}}}.
  \end{align*}
  Both difference quotients are uniformly bounded by $1$ and we distinguish the
  regions based on the size of $\xi 2^{-j_{0}}$ and $\xi 2^{-j_{1}}$.

  If $|\xi| > c 2^{j_{0}}$, we may roughly estimate
  \begin{align*}
    \int_{\{\xi: |\xi| > c 2^{j_{0}}\}} \frac{|\hat{\rho}(\xi)|^{2}}{|\xi|^{2}} \leq \int_{c 2^{j_{0}}}^{\infty} \frac{4}{|\xi|^{2}} d \xi \leq C 2^{-j_{0}},
  \end{align*}
  which is a very small contribution.

  If $|\xi| < c 2^{j_{1}}$ with $c$ small, we may use a Taylor expansion to
  estimate the error of the difference quotient:
  \begin{align*}
    \frac{e^{i\xi 2^{-j_{0}}}-1}{i\xi 2^{-j_{0}}} - \frac{e^{i\xi 2^{-j_{1}}}-1}{i\xi 2^{-j_{1}}} = 1 + \mathcal{O}(\xi 2^{-j_{0}}) -1 + \mathcal{O}(\xi 2^{-j_{1}}) \\
    = \mathcal{O}(\xi 2^{-j_{1}}).
  \end{align*}
  The $H^{-1}$ energy for this segment can hence be estimated by
  \begin{align*}
    \int_{\{\xi: |\xi| < c 2^{j_{1}}\}} \frac{|\hat{\rho}(\xi)|^{2}}{|\xi|^{2}} \leq  C \int_{\{\xi: |\xi| < c 2^{j_{1}}\}} 2^{-2j_{1}} \leq C 2^{-j_{1}}
  \end{align*}
  Finally, if $cj_{1} \leq j \leq c j_{0}$ one difference quotient is about $1$,
  while the other oscillates, but is bounded by $1$. Thus the contribution can
  be estimated as
  \begin{align*}
    \int_{\{\xi: c 2^{j_{1}}\leq |\xi| < c 2^{j_{0}}\}} \frac{|\hat{\rho}(\xi)|^{2}}{|\xi|^{2}} \approx \int_{c2^{j_{1}}}^{c 2^{j_{0}}} \frac{1}{|\xi|^{2}} d\xi \approx 2^{-j_{1}}.
  \end{align*}
\end{proof}

\begin{lem}
  \label{lem:H1bygm}
  Let $\rho \in L^{2}(\R^{n})$ and $r>0$ and define
  $\rho_{r}:=\frac{1}{|B_{r}|}1_{B_{r}}* \rho$. Then $\rho - \rho_{r} \in
  H^{-1}$ and there exists $C>0$ depending only on the dimension $n$ such that
  \begin{align*}
    \|\rho-\rho_{r}\|_{H^{-1}} \leq C r \|\rho\|_{L^{2}}.
  \end{align*}
\end{lem}

\begin{proof}[Proof of Lemma \ref{lem:H1bygm}]
  We consider the Fourier transform of $\phi_{r}$. Let thus $r>0$ and $\xi \in
  \R^{n}$ be given and consider
  \begin{align*}
    \frac{c_{n}}{r^{n}}\int_{B_{r}} e^{ix \cdot \xi} dx = \frac{1}{r^{n}}\int_{B_{r}}e^{i x_{1}|\xi|} = \frac{c_{n}}{|r\xi|^{n}}\int_{B_{r|\xi|}} e^{i x_{1}} dx =: \psi(|r\xi|). 
  \end{align*}
  We remark that $\psi$ can explicitly computed in terms of Bessel functions
  (c.f. the proof of Theorem \ref{thm:lower_geom}). Since $\psi(\cdot)$ is an average of
  $e^{ix_{1}}$ it follows that $|\psi|\leq 1$. Furthermore, by continuity of
  $e^{ix_{1}}$
  \begin{align*}
    |\psi(|r\xi|)-1|\leq \frac{c_{n}}{|r\xi|^{n}}\int_{B_{r|\xi|}} |e^{i x_{1}}-1| dx\leq C r|\xi|.
  \end{align*}
  as $r|\xi| \downarrow 0$.

  Hence, we can control
  \begin{align*}
    |\mathcal{F}(\rho_{r}-\rho)|^{2}= |(\psi(r|\xi|)-1)\hat{\rho}(\xi)|^{2}
    \leq \min(2, C r|\xi|) |\hat{\rho}|^{2}.
  \end{align*}
  and can estimate the $H^{-1}$ energy of $\rho-\phi_{r}*\rho$ by
  \begin{align*}
    \int \frac{\min(C^{2} |\xi r|^{2}, 4) }{|\xi|^{2}} |\hat{\rho}(\xi)|^{2}\leq C^{2}r^{2} \int |\hat{\rho}(\xi)|^{2} = C^{2} r^{2}\|\rho\|_{L^{2}}^{2}.
  \end{align*}
  
\end{proof}

\begin{lem}
  \label{lem:1}
  Let $\rho \in L^{2}(\R^{n})$ with $\|\rho\|_{L^{2}}\leq 1$ be supported in
  $B_{1}(0)$ be such that $\|\frac{1}{c_{n}\epsilon^{n} }1_{B_{\epsilon}} *
  \rho\|_{L^{\infty}} \leq \epsilon$, for some $0< \epsilon < 1$. Then there
  exist $C$ depending only on the dimension $n$ such that the analytic mixing
  scale satisfies
  \begin{align*}
    \|\rho\|_{H^{-1}} \leq C \epsilon.
  \end{align*}
\end{lem}

\begin{proof}[Proof of Lemma \ref{lem:1}]
  By the triangle inequality
  \begin{align*}
    \|\rho\|_{H^{-1}} \leq \|\rho_{\epsilon}\|_{H^{-1}} + \|\rho - \rho_{\epsilon}\|_{H^{-1}}.
  \end{align*}
  The second term can be estimated by $C \epsilon \| \rho \|_{L^{2}}\leq C
  \epsilon$ using Lemma \ref{lem:H1bygm}, while for the first term we control
  \begin{align*}
    \|\rho_{\epsilon}\|_{H^{-1}} \leq \|\rho_{\epsilon}\|_{L^{2}} \leq |\supp(\rho_{\epsilon})| \|\rho_{\epsilon}\|_{L^{\infty}} \leq 2^{n} \epsilon, 
  \end{align*}
  where we estimated the support of $\rho_{\epsilon}$ by $B_{1+\epsilon} \subset
  B_{2}$.
\end{proof}

We remark that since the definition of $\mathfrak{g}$ is given in terms of local
Lebesgue spaces some support or decay condition is necessary.

Indeed, consider $\rho \in L^{2}(\R)$ which is compactly supported in $(0,1)$.
Then for any $N \in \N$ we can define $\sigma(x):= \sum_{j=0}^{N} \rho(x + 2j)$.
Due to the disjoint supports for all $\epsilon<\frac{1}{2}$ it holds that
\begin{align*}
  \|\sigma_{\epsilon}\|_{L^{\infty}}&= \|\rho_{\epsilon}\|_{L^{\infty}}, \\
  \|\sigma\|_{L^{2}} &= \sqrt{N} \|\rho\|_{L^{2}}.
\end{align*}
Furthermore, while $H^{-1}$ is non-local, we obtain that $\|\sigma\|_{H^{-1}}
\approx \sqrt{N} \|\rho\|_{H^{-1}}$ for $N$ large.

The following lemma establishes the converse control of the geometric scale by
the analytic scale. As noted in Remark \ref{rem:regularity_assumptions}, here
regularity of $\phi$ allows for easier proofs. However, under additional
assumptions, $\phi$ can also be chosen less regular such as $1_{B_{1}}$.

\begin{lem}
  \label{lem:withFouriersupport}
  Let $\rho \in H^{-1}(\R^{n})$ with $\supp(\hat{\rho})\subset B_{r^{-1}}(0)$
  and let $\phi \in L^{1}$. Then there exists a constant $C$ depending of $\rho$
  such that
  \begin{align*}
    \mathfrak{g}_{\epsilon_{0}}[\rho] \leq C r^{-n/2-1}\|\rho\|_{H^{-1}}. 
  \end{align*}
  If we require that $\phi \in H^{1}$, then the support assumption can be
  omitted.
\end{lem}
\begin{proof}[Proof of Lemma \ref{lem:withFouriersupport}]
  Using Plancherel's theorem we compute
  \begin{align*}
    \phi_{r}*\rho (x)= \int \phi_{r}(x-y) \rho(y) dy=\int e^{ix \xi} \overline{\hat{\phi_{r}}}(\xi) \hat{\rho}(\xi) d\xi.
  \end{align*}
  Now recall that $\phi_{r}(x)=\frac{\phi(rx)}{r^{n}}$ has constant $L^{1}$ norm
  and thus $\|\hat{\phi_{r}}\|_{L^{\infty}}\leq \|\phi_{r}\|_{L^{1}}$. We may
  hence control further by
  \begin{align*}
    &\quad \|\hat{\phi}_{r}\|_{L^{\infty}} \|\xi\|_{L^{2}(\supp (\hat{\rho}))} \left\|\frac{\hat{\rho}}{\xi}\right\|_{L^{2}}  \\
    & \leq r^{-n/2-1} \|\rho\|_{H^{-1}}, 
  \end{align*}
  where we used the support of $\hat{\rho}$ and that
  $\|\hat{\phi_{r}}\|_{L^{\infty}}\leq \|\phi_{r}\|_{L^{1}}=1$.

  If $\phi \in H^{1}$, we can instead directly estimate 
  \begin{align*}
    \|\phi_{r}*\rho\|_{L^{\infty}} &\leq \|\phi_{r}\|_{H^{1}} \|\rho\|_{H^{-1}} \\
    &= r^{-n/2-1}\|\phi\|_{H^{1}} \|\rho\|_{H^{-1}}.
  \end{align*}
\end{proof}

Hence, the failure of estimates with $s \geq 1/2$ is due to the interaction of
the ``tail'' of $\hat{\rho}$ for $|\xi|\geq r^{-1}$.
We remark that in the dyadic setting of Section \ref{sec:walsh} this
complication does not arise, since our seminorms project on spaces $E_{j}$ of
lower frequency.

\section{Damping Rates in Transport-type Equations}
\label{sec:transport}
In this second part of our article, we are interested in the time dependence of
mixing scales when $\rho(t)$ evolves under a passive scalar equation
\begin{align*}
  \dt \rho + v \cdot \nabla \rho =0
\end{align*}
for a given divergence-free velocity field $v$. In particular, we are interested optimal decay rates of
$\|\rho(t)\|_{H^{-1}}$ and $\mathcal{G}_{\kappa(t)}[\rho(t)]$.

In Subsection \ref{sec:free-transp-equat}, we consider the special case when
$\rho(t)$ is advected by a given specific, regular, incompressible velocity
field. This study is motivated by recent work of Crippa, Luc\`a and Schulze,
\cite{Crippa17}, who study the time behavior of both mixing scales under the
evolution
\begin{align*}
  \rho(t,r,\theta)= \rho_0(r,\theta -tr),
\end{align*}
where $r, \theta$ are polar coordinates on $\R^{2}\setminus \{0\}$, $\rho_{0}
\in L^{1} \cap L^{\infty}$ and the angular averages $\langle \rho_0
\rangle_{\theta}=0$ identically vanish. Adapting conformal polar coordinates
$\theta, e^{s}=r$ this setting shares strong similarities with problems of
inviscid damping in fluid dynamics.

Further examples of interest here are given by:
\begin{itemize}
\item Perturbations around shear flow solutions of Euler's equations on $\T
  \times \R$. In \cite{Zill3}, \cite{Zill5}, we show that if $U(y)$ is, roughly
  speaking, close to affine, the linearized Euler equations in vorticity
  formulation asymptotically scatter in $H^{s}$ to the transport problem with
  $v=(U(y),0)$. Using different, spectral methods \cite{Zhang2015inviscid} have
  further shown similar results under weaker conditions.
\item When considering circular flows, \cite{Zill6}, \cite{CZZ17}, and
  $v=u(r)e_{\theta}$ is an annular region or on $\R^{2}$ we similarly obtain
  stability, damping and scattering in weighted spaces and for more degenerate
  profiles.
\item In the setting of Landau damping \cite{bedrossian2013landau} similarly one
  observes scattering to a transport problem.
\end{itemize}

The following results on the free transport equation
\begin{align*}
  \dt \rho(t,x,y) - y \p_{x} \rho &=0,\\
  (t,x,y) &\in (0,\infty) \times \T^{n} \times \R^{n},
\end{align*}
hence also extend by scattering to asymptotics for further equations exhibiting
\emph{phase-mixing}.

Finally, we discuss optimal mixing and stirring for more general passive scalar
problems. Here, a recent active area of research, \cite{alberti2014exponential},
\cite{seis2017quantitative}, \cite{crippa2017cellular}, \cite{bressan2003lemma},
\cite{crippa2008estimates} is given by the study of upper and lower bounds on
decay rates of mixing scales for solutions of \eqref{eq:8}
\begin{align*}
  \dt \rho + v \cdot \nabla \rho =0,
\end{align*}
where $v$ may be chosen arbitrarily under given constraints such as
$\|v(t)\|_{W^{1,p}}\leq 1$. Using our comparison estimates of Theorem
\ref{thm:continuous}, we discuss implications of some
known results.

\subsection{On Sharp Decay Rates for the Free Transport Equation}
\label{sec:free-transp-equat}

Our main results for the analytic mixing scale are summarized in Theorem
\ref{thm:mixing}, which we restate in the following for easier reference. Using the estimates of Theorem \ref{thm:main} we also obtain control of
the geometric scale, which we study in Theorem \ref{thm:lower_geom}.
\begin{thm}
  Let $H^{\sigma}H^{s}= H^{\sigma}(\T^{n}; H^{s}(\R^{n}))$ denote the Hilbert space with norm
  \begin{align*}
  \|u\|_{H^{\sigma}_{x}H^{s}_{y}}^{2}= \sum_{k \in \Z^{n}} \langle k \rangle^{2\sigma} \int_{\R^{n}} \langle \eta \rangle^{2s}|\tilde{u}(k,\eta)|^{2} d\eta. 
  \end{align*}
  In the following, let $0< s \leq 1$, $u_{0} \in L^{2}(\T^{n}; H^{s}(\R^{n}))$
  with $\int_{\T^{n}} u_{0}(x,y)dx =0$, and let
  \begin{align*}
    u(t,x,y)= u_{0}(t,x-ty,y),
  \end{align*}
  be the solution of the free transport problem.
  \begin{enumerate}
  \item There exists $C_{s}>1$ such that for all $t \geq 1$ and all initial data
    \begin{align*}
      \|u(t)\|_{H^{-1}} \leq C t^{-s}\|u_{0}\|_{H^{-s}(\T^{n}; H^{s}(\R^{n}))}.
    \end{align*}    
  \item Let $\alpha_{j}>0$ with $\|(\alpha_{j})_{j}\|_{l^{2}}=1$. Then there
    exist $c>0$, a sequence of times $t_{j}\rightarrow \infty$ and initial data
    $u_{0}$ such that 
    \begin{align*}
      \|u(t_{j})\|_{H^{-1}} \geq c \alpha_{j} t_{j}^{-s} \|u_{0}\|_{H^{-s}(\T^{n}; H^{s}(\R^{n}))}.
    \end{align*}
  \item There exists no non-trivial initial data $u_{0} \in L^{2}(\T^{n}; H^{s}(\R^{n}))$ such that
    \begin{align*}
      \|u(t_{j})\|_{H^{-1}} \geq c t_{j}^{-s} \|u_{0}\|_{H^{-s}(\T^{n}; H^{s}(\R^{n}))}
    \end{align*}
    along some sequence $t_{j}\rightarrow \infty$.
  \end{enumerate}
\end{thm}
In the second statement, $t_{j}$ can always be chosen larger and more rapidly
increasing. For instance, we may chose $t_{j}=\exp(\exp(\dots \exp(j)))$ and
$\alpha_{j}=\frac{1}{j}=\ln(\ln(\dots \ln(t_{j})))$ as iterated exponential and
logarithms. Informally stated, the theorem hence shows that algebraic decay
rates can be achieved along a subsequence up to an arbitrarily small loss.
Conversely, the third statement shows that this loss is necessary and that the
lower estimate is sharp in this sense.
\begin{proof}[Proof of Theorem \ref{thm:mixing}]
  We note that, for all $t$ the map $L^{2} \ni u_{0}\mapsto u(t) \in L^{2}$ is
  unitary and thus the statement holds for $s=0$. Furthermore, we may use the explicit solution of the free transport
  problem and Plancherel's identity with respect to $x$ to obtain that
  \begin{align*}
    \|u(t)\|_{H^{-1}} &= \sup_{\phi: \|\phi\|_{H^{1}}\leq 1} \sum_{k\neq 0} \int \overline{\hat \phi}(k,y) e^{ikty} \hat{u}_{0}(k,y) dy \\
                      &=  \sup_{\phi: \|\phi\|_{H^{1}}\leq 1} \int e^{ikty}\p_{y} \frac{\overline{\hat \phi}(k,y) \hat{u}_{0}(k,y)}{ikt} \\
    &\leq \|u_{0}\|_{H^{-1}(\T^{n}; H^{1}(\R^{n}))},
  \end{align*}
  and thus establish the result for $s=1$.
  The result for $0<s<1$ then is obtained by interpolation.\\

  For the second statement, we make use of resonant times and frequencies.
  Roughly speaking, if $u_{0}$ is frequency localized at $(k,\eta)$ (with
  respect to $x$ and $y$), then free transport in physical space is also a
  transport equation in Fourier space and $u(t)$ will be frequency localized
  near $(k,\eta+kt)$. Hence, if $k \neq 0$, $\eta$ and $k$ are parallel and $t=-\eta/k$ the frequency
  localization is near zero and hence any $H^{\sigma}$ norm is equivalent to the
  $L^{2}$ norm for such a function.
  
  Let thus $\phi \in C_{c}^{\infty}$ be supported in a ball of radius $2$ and
  $L^{2}$ normalized and let $(\alpha_{j})_{j} \in l^{2}$ with
  $\|(\alpha_{j})_{j}\|_{l^{2}}=1$. Suppose further that $t_{j}$, to be
  determined precisely later, satisfies $\min_{j_{1}\neq
    j_{2}}|t_{j_{1}}-t_{j_{2}}| > 4$ and $\min |t_{j}|>4$. Then we define the
  function $u_{0} \in H^{s}(\T^{n}\times\R^{n})$ by its Fourier transform:
  \begin{align*}
    \tilde{u_{0}}(k,\eta)= \delta_{k=e_{1}}\sum_{j \in \N} \alpha_{j} \langle t_{j} \rangle^{-s} \phi(\eta-t_{j}e_{1}).
  \end{align*}
  We remark that the Dirac in $k$ corresponds to assuming periodicity in $x$.
  This construction also readily extends to the whole space case, $\R^{n} \times
  \R^{n}$ if
  $\delta_{k=e_{1}}\phi(\eta-t_{j}e_{1})$ is replaced by a bump function  $\phi(10 k-e_{1}) \phi(\eta-kt_{j})$.

  By our assumption on $t_{j_{1}}-t_{j_{2}}$, the functions $\phi(\cdot - t_{j})$ are
  disjointly supported and thus
  \begin{align*}
    \|u_{0}\|_{H^{s}}^{2} & = \sum_{j} |\alpha_{j}|^{2} \|\frac{\langle \cdot \rangle^{2s}}{\langle t_{j} \rangle^{2s}} \phi(\cdot-t_{j}e_{1})\|_{L^{2}(B_{2}(0))}^{2} \\
                          &= \sum_{j} |\alpha_{j}|^{2} \|\frac{\langle \cdot-t_{j}e_{1} \rangle^{2s}}{\langle t_{j} \rangle^{2s}} \phi(\cdot)\|_{L^{2}(B_{2}(0))}^{2}.
  \end{align*}
  Similarly, for any $t \in \R$, it holds that
  \begin{align*}
    \|u(t)\|_{H^{s}}^{2} = \sum_{j} |\alpha_{j}|^{2} \|\frac{\langle (t_{j}-t)e_{1}+\cdot \rangle^{2s}}{\langle t_{j} \rangle^{2s}} \phi(\cdot)\|_{L^{2}(B_{2}(0))}^{2}.
  \end{align*}
  By our assumptions on $t_{j}$ and $\phi$, in the first sum $|\eta|\leq 2 \leq
  \frac{|t_{j}|}{2}$, and thus $\|u_{0}\|_{H^{s}}^{2}$ is comparable (within a
  factor $4^{\pm  s}$) to $\sum_{j}|\alpha_{j}|^{2}=1$. One the other hand, for
  $t=t_{j}$, the second sum is bounded from below by
  \begin{align*}
    |\alpha_{j}|^{2}\|\frac{\langle\cdot \rangle^{2s}}{\langle t_{j} \rangle^{2s}} \phi(\cdot)\|_{L^{2}(B_{2}(0))}^{2} \geq \frac{|\alpha_{j}|^{2}}{\langle t_{j} \rangle^{2s}}.
  \end{align*}
  This concludes the proof of the second statement. We note that a similar lower
  bound can also be obtained for the homogeneous Sobolev spaces by choosing
  $t=t_{j}+4$ instead.
  \\
  
  Finally, suppose that there exists $u_{0}$ attaining the algebraic decay
  rates. Expressed in terms of $u_{0}$ this implies that for a sequence
  $t_{j}\rightarrow \infty$
  \begin{align}
    \label{eq:2}
    \|\frac{t_{j}^{s}}{\langle \eta -t_{j}k \rangle^{1}\langle \eta \rangle^{s}}  \langle \eta \rangle^{s} \langle k \rangle^{-s}\tilde{u}_{0}\|_{l^{2} L^{2}}^{2} \geq C  \|\langle \eta \rangle^{s} \langle k \rangle^{-s}\tilde{u}_{0}\|_{l^{2} L^{2}}^{2}.
  \end{align}
  Since the equation decouples with respect to $k$ and for easier notation, in the following we consider
  $k$ arbitrary but fixed and omit the factors $\langle k \rangle^{-s}$.
  The result of the Theorem then follows by multiplying by $\langle k \rangle^{-s}$ and summing
  in $k$.
  
  Now let $t=t_{j}$ and consider the sets $\Omega_{C,t}=\left\{ (k,\eta):
    \frac{|t^{s}|}{\langle \eta -tk \rangle^{1}\langle \eta
      \rangle^{s}}<\sqrt{C/2} \right\}$ and $A_{C,t}$ their complements. Then
  the inequality \eqref{eq:2} implies that 
  \begin{align}
    \label{eq:3}
    \begin{split}
      \|\frac{t^{s}}{\langle \eta -t_{j}k \rangle^{1}\langle \eta \rangle^{s}}  \langle \eta \rangle^{s} \tilde{u}_{0}\|_{L^{2}(A_{C,t})}^{2} + \frac{C}{2}\| \langle \eta \rangle^{s} \tilde{u}_{0}\|_{L^{2}(\Omega_{C,t})}^{2} \geq C  \| \langle \eta \rangle^{s} \tilde{u}_{0}\|_{L^{2}(\Omega_{C,t})}^{2} \\
      \Rightarrow \|\frac{t^{s}}{\langle \eta -kt_{j} \rangle^{1}\langle \eta
        \rangle^{s}} \langle \eta \rangle^{s}
      \tilde{u}_{0}\|_{L^{2}(A_{C,t})}^{2} \geq C/2 \| \langle \eta \rangle^{s}
      \tilde{u}_{0}\|_{L^{2}}^{2}
    \end{split}
  \end{align}
  On the other hand
  \begin{align*}
    \frac{t^{s}}{\langle \eta -kt_{j} \rangle^{1}\langle \eta \rangle^{s}} \leq \max\left( \frac{t^{s}}{\langle kt_{j}/2 \rangle^{1} 1}, \frac{t^{s}}{1 \langle kt_{j}/2 \rangle^{s}} \right)\leq 2^{s} ,
  \end{align*}
  by considering $|\eta|\leq t_{j}/2$ and $|\eta|>t_{j}/2$ and using that $0 \leq s \leq
  1$.

  Hence, it follows that, for a constant depending on $s$, but independent of
  $t$,
  \begin{align*}
    \| \langle \eta \rangle^{s} \tilde{u}_{0}\|_{L^{2}(A_{C,t})}^{2} \geq C_{s} \|\langle \eta \rangle^{s} \tilde{u}_{0}\|_{L^{2}}^{2}.
  \end{align*}
  By assumption, this holds for a sequence $t_{j}\rightarrow \infty$. Upon
  passing to a subsequence, the sets $A_{C,t_{j}}$ can be ensured to be mutually
  disjoint. Hence, by orthogonality
  \begin{align*}
    \|\langle \eta \rangle^{s} \tilde{u}_{0}\|_{L^{2}}^{2} \geq \sum_{j} \| \langle \eta \rangle^{s} \tilde{u}_{0}\|_{L^{2}(A_{C,t_{j}})}^{2} \\
    \geq C_{s}\sum_{j}  \|\langle \eta \rangle^{s} \tilde{u}_{0}\|_{L^{2}}^{2} = \infty \|\langle \eta \rangle^{s} \tilde{u}_{0}\|_{L^{2}}^{2} .
  \end{align*}
  This is a contradiction unless $v=0$.
\end{proof}

As a consequence of our comparison result, Theorem \ref{thm:continuous}, we also
obtain lower bounds on the decay of the geometric mixing scale.
The following theorem instead provides a direct construction of a lower bound, 
where averages are taken over the scale $\frac{j}{t_{j}}$ instead.

\begin{thm}
  \label{thm:lower_geom}
  Let $0\leq s <\frac{1}{2}$, then there exists initial data $u_{0} \in
  L^{2}_{x}H^{s}_{y}(\T \times [0,2 \pi])$, so that along a sequence
  of times $t_{j}=2^{100+j}$ the solution $u$ of the free transport problem
  satisfies
  \begin{align*}
    \mathfrak{g}_{\frac{j}{t_{j}}} [u]\geq \frac{\|u_{0}\|_{L^{2}H^{s}}}{j \langle t_{j} \rangle^{s}}.
  \end{align*}
  That is, at scale $r=\frac{j}{t_{j}}$ we have a lower
  bound by $\frac{1}{j} \frac{1}{t_{j}^{s}}$.
\end{thm}
We remark that as in the previous Theorem \ref{thm:mixing}, $t_{j}$ can be
chosen to be increasing more rapidly and thus $\frac{1}{j}=o(1)$ as
$t_{j}\rightarrow \infty$ can be
chosen with very slow decay. Furthermore, the construction of our proof also
extends to the $n$ dimensional transport equations by extending
constantly in the other directions.

\begin{proof}[Proof of Theorem \ref{thm:lower_geom}]
  Consider the function
  \begin{align*}
    u_{0}(x,y)=c e^{ix}\sum_{j=1}^{\infty}\frac{1}{j} \frac{e^{i t_{j}y}}{\langle t_{j} \rangle^{s}},
  \end{align*}
  as a functions on $\T \times [0,2\pi]$, where $c \in (\frac{1}{100},100)$ can
  be chosen such that this function is normalized since $\frac{1}{j} \in l^{2}$.

  Then, the solution $u$ of the free transport problem is explicitly given by
  \begin{align*}
    u(t,x,y)=c e^{ix}\sum_{j=1}^{\infty}\frac{1}{j} \frac{e^{i(t_{j}-t)y}}{\langle t_{j} \rangle^{s}},
  \end{align*}

  For simplicity of notation and presentation we first consider averages over
  squares instead of balls, which allows for a simpler straightforward
  calculation. An extension to the latter setting is given at the end of this
  proof. Let thus $t=t_{j_{0}}$ and consider a square $S=I_{x} \times I_{y}$ of
  side length $\frac{1}{100}> d>\frac{j_{0}}{t_{j_{0}}}$, which is centered
  close to a point where $e^{ix}=1$ Then the integral of $u$ over this square
  decouples by Fubini's theorem and we may compute that
  \begin{align*}
    \frac{1}{d}\int_{I_{x}} e^{ix} dx \approx 1
  \end{align*}
  and that
  \begin{align*}
    &\quad \frac{1}{d}\int_{I_{y}}  \sum_{j=1}^{\infty}\frac{1}{j} \frac{\cos((t_{j}-t)y)}{\langle t_{j} \rangle^{s}} dy
    \\ &= \frac{1}{d}\int_{I_{y}} \frac{1}{j_{0}\langle t_{j} \rangle^{s}} + \frac{1}{d} \sum_{j \neq j_{0}} \int_{I_{y}} \frac{1}{j} \frac{e^{i(t_{j}-t_{j_{0}})y)}}{\langle t_{j} \rangle^{s}}.
  \end{align*}
  We note that as an average over a constant function, the first term equals
  \begin{align}
    \label{eq:1}
    \frac{1}{j_{0}\langle t_{j} \rangle^{s}} \approx \frac{1}{j_{0}2^{j_{0}s}}.
  \end{align}
  On the other hand, by construction of $t_{j}=2^{100+j}$, for each $j \neq
  j_{0}$, $|t_{j}-t_{j_{0}}| \geq \frac{1}{2} \max (t_{j}, t_{j_{0}})$ and thus
  all further integrands are highly oscillatory. In particular,
  \begin{align*}
    \left| \int_{S_{y}} \frac{1}{j} \frac{\cos((t_{j}-t_{j_{0}})y)}{\langle t_{j} \rangle^{s}} \right|
    \leq \frac{1}{j \langle t_{j} \rangle^{s} |t_{j}-t_{j_{0})}|} \leq \frac{1}{j 2^{js} |2^{j}-2^{j_{0}}|}.
  \end{align*}
  where we used integration by parts. Considering the sum in $j$, we split into
  \begin{align*}
    \sum_{j<j_{0}} \frac{1}{j 2^{js} |2^{j}-2^{j_{0}}|} \leq \frac{2}{2^{j_{0}}} \sum_{j}\frac{1}{j 2^{js}} \leq \frac{c_{s}}{2^{j_{0}}}.
  \end{align*}
  and
  \begin{align*}
    \sum_{j>j_{0}} \frac{1}{j 2^{js} |2^{j}-2^{j_{0}}|} \leq 2^{-j_{0}s} \sum _{j>j_{0}}\frac{2}{j 2^{j}}  \leq 2^{-j_{0}s} 2^{-j_{0}}. 
  \end{align*}
  Dividing by $d>\frac{j2^{-j_{0}}}{100 c_{s}}$, both terms will be smaller than
  the term in \eqref{eq:1} by a large factor and hence
  \begin{align*}
    \frac{1}{|S|} \int_{S} u(t_{j}) \geq c \frac{1}{j_{0}\langle t_{j} \rangle^{s}},
  \end{align*}
  as claimed.
  \\

  Let us next consider the original problem of averages over balls. In this case
  the integrals
  \begin{align*}
    \frac{1}{j\langle t_{j} \rangle^{s}}  \frac{1}{|B_{d}|}\int_{B_{d}}  e^{ix} e^{i(t_{j}-t)y}
  \end{align*}
  can be explicitly computed in terms of Bessel functions. That is, if the
  center of the ball is the point $(\xi_{1}, \xi_{2})$, then after translating
  in $x$ and $y$, we obtain an exponential factor $e^{i\xi_{1}+
    i(t_{j}-t)\xi_{2}}$ and an integral over a ball centered in $(0,0)$. We
  hence, need to compute
  \begin{align*}
    \int_{B_{d}} e^{ix (1, t_j-t)} dx = d^{n} \int_{B_{1}}e^{ix (d, d(t_{j}-t))} dx.
  \end{align*}
  That is, the Fourier transform of the indicator function of a ball.

  Using the rotation-invariance of $B_{1}$, we compute
  \begin{align*}
    & \quad \int_{B_{1}} e^{ix \xi} dx = \int_{B_{1}} e^{ix_{1}|\xi|} dx_{1}dx_{2} \\
    &= \int_{-1}^{1} \sqrt{1-x_{1}^{2}} e^{ix_{1}|\xi|} dx_{1} 
      =c \frac{J_{1}(|\xi|)}{|\xi|},
  \end{align*}
  where $J_{1}$ denotes the Bessel function of the first kind and $c\leq 10$. It
  hence follows that
  \begin{align*}
    d^{-n}\int_{B_{d}} e^{ix (1, t_j-t)} dx \leq C_{d} \min \left(1, \frac{C}{d\sqrt{1+|t_{j}-t|^{2}}}\right).
  \end{align*}
  and that
  \begin{align*}
    d^{-n}\int_{B_{d}} e^{ix (1, t_j-t)} dx \approx 1 
  \end{align*}
  if $t_{j}-t$ is small. Thus, the above estimates for squares extend to this
  case in a straightforward way.
\end{proof}

\subsection{On Lower Bounds for Mixing Costs}
\label{sec:mixing_costs}
Consider again the passive scalar problem
\begin{align}
  \label{eq:10}
  \begin{split}
  \dt \rho + v\cdot \nabla \rho&=0,\\ \nabla \cdot \rho&=0.
  \end{split}
\end{align}
In the previous section we considered $v$ as given and asked about decay rates
of mixing scales for $\rho_{0} \in H^{s}$ to be chosen freely.

As a related and in a sense inverse problem, one can ask about \emph{mixing
  costs}. That is, you are given an explicit initial datum $\rho_{0} \in H^{s}$
and want it to be mixed to scale $\epsilon$ by time $1$. What kind of lower
bound does this imply on Sobolev norms of $v$ in space and time?

More precisely, the aim is to a establish a lower bound of the type
\begin{align*}
  \int_{0}^{1} \|v\|_{W^{1,p}}dt \geq C_{p} |\log(\epsilon)|,
\end{align*}
when $\rho(1)$ is geometrically mixed to scale $\epsilon$.
The case $p>1$ has been established in \cite{crippa2008estimates} and
the case $p=1$ is a conjecture of Bressan, \cite{bressan2003lemma}.

As an application of our comparison results, we consider the simplest case of
$p=\infty$, following the proof in \cite{crippa2008estimates} via Gronwall's
estimate.

\begin{lem}
  \label{lem:Linfty}
  Let $1>\epsilon>0$ and $\rho$ be a solution of \eqref{eq:10} on $R^{n}$ be such that
  \begin{align*}
    \|\rho|_{t=0}\|_{H^{-1}}=1, \quad \|\rho|_{t=1}\|_{H^{-1}}=\epsilon>0,
  \end{align*}
  with $v \in W^{1,\infty}$.
  Then it holds that
  \begin{align*}
    \int_{0}^{1} \|\p_{i}v_{j}+\p_{j}v_{i}\|_{L^{\infty}} dt \geq C |\log(\epsilon)|.
  \end{align*}
\end{lem}
\begin{proof}
  Since $v$ is divergence-free the solution operator $S(t_{2},t_{1})$ mapping
  $\rho|_{t_{1}}$ to $\rho|_{t_{2}}$ is unitary.
  Characterizing the $H^{-1}$ norm via duality we hence obtain that
  \begin{align*}
    \|\rho|_{t=0}\|_{H^{-1}} = \sup_{\|\psi\|_{H^{1}}\leq} \int \psi \rho_{t=0} = \sup_{\|\psi\|_{H^{1}}\leq}\int \psi \rho_{t=0} \\
    = \sup_{\|\psi\|_{H^{1}}\leq}\int (S(1,0) \psi)\rho_{t=1} \leq \|S(1,0)\|_{H^{1}\rightarrow H^{1}\|} \|\rho_{t=1}\|_{H^{-1}},
  \end{align*}
  and thus
  \begin{align}
    \label{eq:11}
    \|S(1,0)\|_{H^{1}\rightarrow H^{1}\|} \geq \frac{1}{\epsilon}.
  \end{align}
  On the other hand, $S(1,0)$ conserves the $L^{2}$ norm and $\p_{j} \rho$
  satisfies
  \begin{align*}
    \dt \p_{j}\rho + v \cdot \nabla \p_{j} \rho + (\p_{j}v)\cdot \nabla \rho =0. 
  \end{align*}
  Testing against $\rho_{j}$ we thus obtain that
  \begin{align*}
    \frac{d}{dt}\|\nabla \rho\|_{L^{2}}^{2} \leq 2 \sum_{i,j} \int  (\p_{j}\rho) (\p_{j}v_{i}) (\p_{i}\rho) \\
    = \sum_{i,j} \int  (\p_{j}\rho) (\p_{j}v_{i}+\p_{i}v_{j}) (\p_{i}\rho) \leq \|\p_{j}v_{i}+\p_{i}v_{j}\|_{L^{\infty}} \|\nabla \rho\|_{L^{2}}^{2}.
  \end{align*}
  Gronwall's inequality thus implies that
  \begin{align*}
    \|S(1,0)\|_{H^{1}\rightarrow H^{1}} \leq \exp \left( \int_{0}^{1} \|\p_{j}v_{i}+\p_{i}v_{j}\|_{L^{\infty}} dt \right),
  \end{align*}
  which in combination with the inequality \eqref{eq:11} concludes the proof.
\end{proof}

As a corollary we obtain a lower bound on the geometric scale. While our
comparison estimates of Section \ref{sec:proofs} can not be expected to be
optimal due the different time dependence, we remark that lower bounds in terms
of powers of $\epsilon$ yield the same logarithmic lower bounds.
Hence, we may consider the assumptions of the following corollary to be
equivalent to those of Lemma \ref{lem:Linfty} for our purposes.
\begin{cor}
  Let $1>\epsilon >0$ and $\rho$ be a solution of \eqref{eq:10} on $\R^{n}$ with
  $v \in W^{1,\infty}(\R^{n})$.
  Suppose that $\|\rho|_{t=0}\|_{H^{-1}}=1$ and that $\rho|_{t=1}$ is
  supported in $B_{1}$ and
  \begin{align*}
    \mathcal{G}_{\epsilon}[\rho|_{t=1}]\leq \epsilon.
  \end{align*}
  Then it follows that
  \begin{align*}
     \int_{0}^{1} \|\p_{i}v_{j}+\p_{j}v_{i}\|_{L^{\infty}} dt \geq C |\log(\epsilon)|.
  \end{align*}
\end{cor}
\begin{proof}
  By Theorem \ref{thm:continuous}, $\rho$ also satisfies the assumptions of
  Lemma \ref{lem:Linfty}, which implies the result.
\end{proof}

For the case $p>1$, in \cite{crippa2008estimates} Crippa and De Lellis obtain the following mixing
cost result. Unlike the setting $p=\infty$ this seminal result requires
considerable effort to prove. In subsequent works we intend to study whether the
comparability can be used to simplify steps of this proof. For this article, we
only state a simple corollary of the established results.
\begin{thm}[Theorem 6.2 in \cite{crippa2008estimates}]
  Let $p>1$ and $\rho|_{t=0}= 1_{[0,1/2]}(x_{2}) \in L^{1}(\T^{2})$ and suppose
  that for $\epsilon>0$ and some $0< \kappa < \frac{1}{2}$ the solution of \eqref{eq:10} satisfies
  \begin{align*}
    \kappa \leq \frac{1}{B_{\epsilon}}\int_{B_{\epsilon}} \rho|_{t=1} \leq 1-\kappa. 
  \end{align*}
  Then there exits a constant $C$ such that
  \begin{align}
    \label{eq:12}
    \int_{0}^{1}\|\nabla v \|_{L^{p}} dt \geq C |\log(\epsilon)|
  \end{align}
  for every $0<\epsilon<\frac{1}{4}$.
\end{thm}

\begin{cor}
  \label{cor:mixing_cost}
  Let $p>1$ and $\rho|_{t=0}= 1_{[0,1/2]}(x_{2}) \in L^{1}(\T^{2})$  and suppose
  that for $\epsilon>0$ and some $0< \kappa < \frac{1}{2}$ the solution of \eqref{eq:10} satisfies
  \begin{align*}
    \|\rho|_{t=1}-\frac{1}{2}\|_{H^{-1}} \leq \epsilon.
  \end{align*}
  Then inequality \eqref{eq:12} holds.
\end{cor}

  \begin{proof}
  Theorem \ref{thm:main} implies that for $0<\alpha<\frac{1}{2}$
  \begin{align*}
    \left| \frac{1}{B_{\epsilon^{\alpha}}}\int_{B_{\epsilon^{\alpha}}} \rho|_{t=1} -\frac{1}{2} \right| \leq C \epsilon^{1/2-\alpha},
  \end{align*}
  where the upper bound on $\alpha$ is due to the regularity of $1_{B_{1}}$ as
  discussed in Remark \ref{rem:regularity_assumptions}.
  Denoting $\delta:=C \epsilon^{1/2-\alpha}$ and adding $\frac{1}{2}$, we thus
  obtain
  \begin{align*}
    \frac{1}{2}-\delta \leq \frac{1}{B_{\epsilon^{\alpha}}}\int_{B_{\epsilon^{\alpha}}} \rho|_{t=1} \leq \frac{1}{2} +\delta.
  \end{align*}
  Thus, we may apply the Theorem of Crippa and De Lellis with
  $\kappa \leq \frac{1}{2}-\delta$ and $\epsilon^{\alpha}$ to obtain that
  \begin{align*}
    \int_{0}^{1}\|\nabla v \|_{L^{p}} dt \geq C |\log(\epsilon^{\alpha})| = C \alpha  |\log(\epsilon^{\alpha})| .
  \end{align*}
    
  \end{proof}

\bibliographystyle{alpha}
\bibliography{citations,citations1}

\end{document}